\documentclass[onefignum,onetabnum]{siamart190516}

\usepackage{amsmath}

\usepackage{amsfonts, psfrag, graphicx, amssymb,enumitem,indentfirst,float,geometry,color,enumerate,graphicx,subcaption,algorithm,algpseudocode,pifont,hyperref,mathtools,mathrsfs}
\graphicspath{{./Figure/}}
\usepackage{mathbbol}
\usepackage{pbox}
\usepackage{booktabs, cellspace, hhline}
\allowdisplaybreaks[4]
\numberwithin{equation}{section}
\numberwithin{figure}{section}

\usepackage{etoolbox}
\patchcmd{\thebibliography}{\chapter*}{\section*}{}{}
\usepackage[title]{appendix}

\usepackage{lipsum}

\newsiamremark{assumption}{Assumption}
\newsiamremark{remark}{Remark}

\newcommand{\commentout}[1]{{}} 

\newcommand{\abs}[1]{\left|#1\right|}

\newcommand{\bfa}{{\bf a}}

\newcommand{\bfH}{{\bf H}}

\newcommand{\bfn}{{\bf n}}

\newcommand{\bfP}{{\bf P}}
\newcommand{\bfp}{{\bf p}}

\newcommand{\bft}{{\bf t}}

\newcommand{\bfu}{{\bf u}}

\newcommand{\bfv}{{\bf v}}

\newcommand{\bfx}{{\bf x}}

\newcommand{\bfz}{{\bf z}}

\newcommand{\conv}{\operatorname{conv}}
\newcommand{\ddiv}{\operatorname{div}}
\newcommand{\rot}{\operatorname{rot}}

\newcommand{\jump}[1]{\left[\!\left[#1\right]\!\right]}
\newcommand{\dd}{\,{\rm d}}

\newcommand{\vertiii}[1]{{\left\vert\kern-0.25ex\left\vert\kern-0.25ex\left\vert #1
    \right\vert\kern-0.25ex\right\vert\kern-0.25ex\right\vert}}
    \newcommand{\vertii}[1]{{\left\vert\kern-0.25ex\left\vert #1
    \right\vert\kern-0.25ex\right\vert}}

\begin{document}
\title{
On the maximum angle conditions for polyhedra\\
 with virtual element methods
}
\author{Ruchi Guo \thanks{Department of Mathematics, University of California, Irvine, CA 92697 (ruchig@uci.edu).} 
\funding{This work was funded in part by NSF grant DMS-2012465.}
}
\date{}
\maketitle
\begin{abstract}
Finite element methods are well-known to admit robust optimal convergence on simplicial meshes satisfying the maximum angle conditions.
But how to generalize this condition to polyhedra is unknown in the literature. 
In this work, we argue that this generation is possible for virtual element methods (VEMs).
In particular, we develop an anisotropic analysis framework for VEMs where the virtual spaces and projection spaces remain abstract and can be problem-adapted, 
carrying forward the ``virtual'' spirit of VEMs. 
Three anisotropic cases will be analyzed under this framework: 
(1) elements only contain non-shrinking inscribed balls but \textit{are not necessarily star convex} to those balls; 
(2) elements are cut arbitrarily from a background Cartesian mesh, which can extremely shrink; 
(3) elements contain different materials on which the virtual spaces involve discontinuous coefficients. 
The error estimates are guaranteed to be independent of polyhedral element shapes. 
The present work largely improves the current theoretical results in the literature and also broadens the scope of the application of VEMs.
\end{abstract}



\section{Introduction}

Polyhedral meshes admit many attractive features, especially the flexibility to adapt to complex geometry. For existing works on polyhedral meshes, say discontinuous Galerkin methods \cite{2012BassiBottiColombo,2013AntoniettiGianiHouston,2022CangianiDongGeorgoulis,2017CangianiDongGeorgoulisHouston,2014CangianiGeorgoulisHouston,2020PietroDroniou}, mimetic finite difference methods \cite{2005BrezziLipnikovShashkov,2005BREZZILIPNIKOVSIMONCINI,2010VeigaManzini}, weak Galerkin methods \cite{2012MuWwangYe,2014WangYe} and virtual element methods (VEMs) \cite{BEIRAODAVEIGA20171110,Brenner;Sung:2018Virtual,2017VeigaCarloAlessandro,2013BeiraodeVeigaBrezziCangiani,2017AndreaEmmanuilSutton} to be discussed, their assumptions for element shape eventually turn into the conventional shape regularity if the polyhedral meshes reduce to simplicial meshes. However, for simplicial meshes, it is well-known that the robust optimal convergence of  finite element methods (FEMs) can be achieved merely under the \textit{maximum angle conditions} \cite{1976BabuskaAziz,1999Duran,2020KobayashiTsuchiya,1992Michal,1994Shenk} which allows extremely narrow and thin elements.
However, how to extend such a condition to polyhedral elements is far from straightforward and still remains open. 
In this work, with VEMs we demonstrate that the extension can be achieved by assuming the existence of a boundary triangulation on polygonal faces of elements satisfying the 2D maximum angle condition and, roughly speaking, of the three edges not nearly coplanar, see Remark \ref{rem_tetangle} for the explanation.



The VEMs were first introduced in  \cite{2013BeiraodeVeigaBrezziCangiani}, 
and the key idea is to develop local problems to construct virtual spaces for approximation. 
VEMs possess several attractive features, especially the conformity to the underlying Hilbert spaces and the flexibility to handle arbitrary polygonal or polyhedra element shapes. 
Those features bring numerous applications of VEMs in many fields. 
For instances, in \cite{2014BenedettoPieraccini,2022BenvenutiChiozziManziniSukumar,2017ChenWeiWen,2018ThanhZhuangXuan,2018HusseinBlaFadi}, 
VEMs are used on meshes cut by interfaces, fractures and cracks, and the convergence is robust with respect to these highly anisotropic meshes, which benefits mesh generation procedure in these problems. 
In fact, it was observed in \cite{BEIRAODAVEIGA20171110} that the optimal convergence of VEMs can be achieved on Voronoi meshes of which the control vertices are randomly generated. 
The robustness to element shapes also benefits solving multiscale problems \cite{2021XieWangFeng,2019RivarolaBenedettoLabanda,2020SreekumarTriantafyllouBcot}. 
In \cite{2021CaoChenGuoIVEM,2022CaoChenGuo}, the authors developed virtual spaces of which the local problems involve discontinuous coefficients, which can be used for unfitted mesh methods of interface problems. 
However, the error analysis for VEMs on anisotropic meshes seems quite challenging. 
Most of the earlier works \cite{2017VeigaCarloAlessandro,2013BeiraodeVeigaBrezziCangiani,2017AndreaEmmanuilSutton} assume the following shape regularity: 
an element $K$ together with each of its faces and edges are all star-convex to a non-shrinking ball in 3D, 2D and 1D, respectively; 
namely the radius of the ball is $\mathcal{O}(h_K)$. In another word, this assumption rules out short edges, small faces, and shrinking elements.


Some efforts have been made to relax the shape conditions.
As a fundamental advance, the ``no small face" and ``no short edge" assumptions were relaxed by Brenner et al. in \cite{Brenner;Sung:2018Virtual} 
in which the balls for star convexity on faces and edges, respectively, are allowed to have the radius $\mathcal{O}(h_F)$ and $\mathcal{O}(h_e)$ instead of $\mathcal{O}(h_K)$. 
The latter one trivially means the edges can be arbitrarily small, while the former one means that a face $F$ can be small but cannot be thin, say thin rectangles. 
However, as the price to pay, their estimates involve an unfavorable factor in the error bound: $\ln(1 + \max_F \tau_F )$,
with $\tau_F$ being the ratio of the longest edge and shortest edge. What's more, the analysis in \cite{Brenner;Sung:2018Virtual} largely relies on the star convexity of elements.
In the 2D case, these conditions are completely circumvented by the approach in \cite{2018CaoChen} through a specially-designed stabilization, without introducing the ``$\ln$" factor. 
A similar work in \cite{2021CaoChenGuoIVEM} extends the approach to virtual spaces with discontinuous coefficients, still in 2D. 
For 3D VEMs, the authors in \cite{Cao;Chen:2018AnisotropicNC} considered a nonconforming VEM without the star-convexity assumptions of faces and edges, 
and they proposed a ``height condition" to replace the star convexity of $K$, i.e., the height of a face $F$ towards its neighbor element $K$ must be $\mathcal{O}(h_F)$. 
This condition follows the one in \cite{2014WangYe} which is a bit more restrictive, i.e., the height must be $\mathcal{O}(h_K)$. 
However, in \cite{Cao;Chen:2018AnisotropicNC}, the method relies on a projection $\Pi_{\omega_K}$ defined onto the whole patch $\omega_K$ of an element $K$ 
whose computation is inevitably complex and expensive. 
It is also worthwhile to mention that, many of these works can only handle the energy norm estimates, as the ``$\ln$"  factor seems difficult to be removed when estimating the $L^2$ errors.

In summary, many works have been devoted to the 3D anisotropic analysis of VEMs, yet the problem still remains quite open.
In this work, we develop a framework for the anisotropic analysis of VEMs for the following model problem: find $u\in H^1_0(\Omega)$
\begin{subequations}
\label{model}
\begin{align}
\label{inter_PDE}
a(u,v) := (\beta \nabla u, \nabla v)_{\Omega} = (f,v)_{\Omega}, ~~~ \forall v\in H^1_0(\Omega)
\end{align}
\end{subequations}
where $\Omega$ is a 3D domain, $f\in L^2(\Omega)$, and $\beta$ will be specified later.
 We show the robust optimal convergence for VEMs on three types of meshes: 
 (1) elements contain \textit{but not necessarily star-convex} to non-shrinking balls; 
 (2) elements are cut arbitrarily from a background Cartesian mesh which may extremely shrink; 
 (3) elements contain different materials on which the virtual spaces involve discontinuous coefficients.  
 We illustrate in Figures \ref{fig:anisotropic_element}-\ref{fig:cub_example} for the three cases. 
 Note that the aforementioned works in the literature cannot cover either of these three cases, as there is no star convexity condition or height condition anymore. 
 The presented analysis can benefit many applications of VEMs for interface, fracture and crack problems \cite{2014BenedettoPieraccini,2022BenvenutiChiozziManziniSukumar,2017ChenWeiWen,2018HusseinBlaFadi,2018ThanhZhuangXuan}.
 In addition, it is highlighted that all the three cases above allow short edges and small faces, and the case (2) even allows very thin elements.


\begin{figure}[h]
  \centering
    \begin{minipage}{0.45\textwidth}
  \centering
   \includegraphics[width=1.5in]{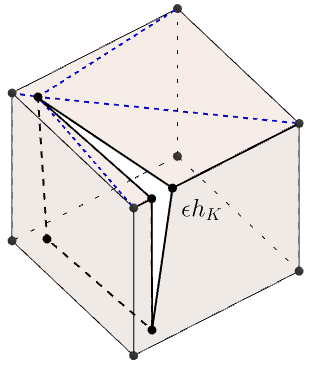}
  \caption{Example of an anisotropic element $K$ which contains a ball of the radius $\mathcal{O}(h_K)$ but is not star convex to it. 
  The left face is not supported by $\mathcal{O}(h_K)$ height towards $K$.
  But its boundary admits a triangulation satisfying the maximum angle condition specified by Assumptions \hyperref[asp:A1]{A1} and \hyperref[asp:A2]{A2}. 
  This element corresponds to Case (1) studied in this paper.}
  \label{fig:anisotropic_element}
  \end{minipage}
 ~~~ 
  \begin{minipage}{0.45\textwidth}
  \centering
   \includegraphics[width=1.5in]{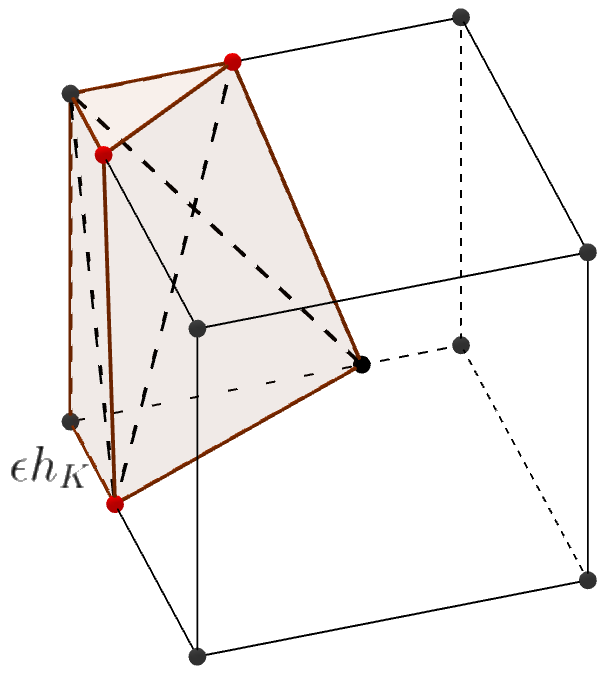}
  \caption{Example of a possibly very thin element $K$ cut from a cub ($\epsilon \rightarrow 0$). 
  It does not even contain a ball of the radius $\mathcal{O}(h_K)$. 
  In addition, the interior dihedral angle $A_5$-$D_1D_4$-$D_2$ may approach $\pi$ such that the 3D triangulation does not satisfy the maximum angle condition, but the boundary triangulation still has the bounded maximum angle.
  This element corresponds to Case (2) studied in this paper. 
  We can also treat the two subelements of the cub together as a entire polyhedron even if they contain discontinuous PDE coefficients. 
  Then, this cube corresponds to Case (3) studied in this paper.}
  \label{fig:cub_example}
  \end{minipage}
\end{figure}

The analysis of the three anisotropic cases described above can facilitate many applications of VEMs. 
For instance, the case (1) may appear for simulating crack propagation with a background shape regular meshes \cite{2014BenedettoPieraccini,2022BenvenutiChiozziManziniSukumar,2018HusseinBlaFadi}, see Figure \ref{fig:anisotropic_element} for an example.
The case (2) may appear when solving interface problems on a background Cartesian mesh 
with a fitted mesh formulation \cite{2017ChenWeiWen,2009ChenXiaoZhang,1998ChenZou}, while the case (3) is for an unfitted mesh formulation \cite{2020AdjeridBabukaGuoLin,2015BurmanClaus,2021CaoChenGuoIVEM,2022CaoChenGuo,2020GuoLin}. See Figure \ref{fig:cub_example} for illustration.
It is also worthwhile to mention that, in most of these applications, the truly anisotropic elements are generally merely concentrated around a 2D manifold in the 3D domain, 
and thus a block-diagonal smoother can alleviate the ill-conditioning issue caused by the irregular element shapes, see \cite{2022CaoChenGuo,2021HuWang}. 
Hence, the application of VEMs to these problems can greatly facilitate effective and efficient computation, which is particularly advantageous for the problems of moving interfaces \cite{2020Guo,2018GuoLinLinElasto,2022MaZhangZheng} and growing cracks \cite{2001DolbowMoesBelytschko,1999MoesDolbowBelytschko}. 


To achieve the robust optimal error estimates, there are two major theoretical innovations in this work. 
The first one is to build the maximum angle condition for general polyhedra through the boundary triangulation. 
This trick, first introduced in \cite{2017ChenWeiWen}, follows from the practice that a 2D triangulation is generally easy, though 3D triangulation can be much more difficult.
In fact, it is trivial that a 3D triangulation satisfying the maximum angle condition must result in a 2D triangulation on the polyhedron boundary with the bounded maximum angle.
However, the converse is not true in general even for a simple prism, see Figure \ref{fig:cub_example} for a counter-example and see Section \ref{sec:fitted _2} for a more detailed discussion of the geometry.
Thus, the present work mainly requiring 2D triangulation on the boundary can largely reduce the complexity of mesh generation, see Figures \ref{fig:anisotropic_element} and \ref{fig:cub_example}, for example.
In addition, we refer readers to \cite{2021CaoChenGuo} for using a virtual mesh satisfying the maximum angle condition to show anisotropic analysis of VEMs.

The second one is a unified and systematic analysis framework for VEMs. 
Under this framework, build on the `virtual" spirit of VEMs, the virtual spaces, local problems and spaces for projection remain abstract and can be adapted to the underlying problems' nature in various applications. 
For instance, the local problems can involve singular coefficients, 
and the spaces for the projection are free of polynomials and can be chosen as any computable spaces as long as they can provide sufficient approximation capability. 
This is particularly advantageous for $\beta$ in \eqref{model} being non-smooth functions, in that the spaces themselves already encode the singularity information. 
The idea here is closely related to multiscale FEMs \cite{1997HouWu,2010ChuGrahamHou} and generalized FEMs \cite{1983BabuskaOsborn,1994BabuskaCalozOsborn}.

To sum up, the proposed analysis is largely different from those in the literature \cite{2013BeiraodeVeigaBrezziCangiani,Brenner;Sung:2018Virtual} 
which, roughly speaking, all try to estimate $u-u_I$ by establishing an affine mapping to certain reference elements or balls, where $u_I$ is the virtual interpolation. Such an affine mapping leads to many critical theoretical tools including the trace inequality, Poincar\'e-type inequality, interpolation estimates, etc. In this work, we will employ a completly different way to show these inequalities by establishing a group of delicate results regarding the maximum angle conditions, which, to our best knowledge, have not appeared in any literature.
Another highlight is that the classical error estimates based on the maximum angle conditions requiring relatively higher regularity \cite{1999Duran,1994Shenk} are circumvented in this work.

This article consists of 5 additional sections. 
In the next section, we present an abstract framework for VEMs by introducing several so-called hypotheses and showing these hypotheses can lead to optimal errors. 
In Section \ref{sec:boundary}, we establish the face triangulation and boundary spaces. 
In the next 3 sections, we analyze the aforementioned 3 types of VEMs by verifying the hypotheses. 
Some technical results will be given in the Appendix.




\section{A unified framework}
\label{sec:unifyframe}

Throughout this article, let $\mathcal{T}_h$ be a polyhedral mesh of $\Omega$,
let $h_K$ be the diameter of an element $K$, and define $h=\max_{K}h_K$. 
We denote the collection of faces and edges of an element $K$ as $\mathcal{F}_K$ and $\mathcal{E}_K$. 
The mesh $\mathcal{T}_h$ is allowed to have shrinking edges, faces and elements, and the geometrical conditions are left for later discussions in detail. 
Let $H^m(D)$ be the standard Sobolev space on a region $D$ and let $H^1_0(D)$ be the space with the zero trace on $\partial D$. 
We further let $\bfH(\ddiv;D)=\{\bfu\in \bfH(D), \ddiv(u)\in H(D)\}$. 
In addition, $\|\cdot\|_{m,D}$ and $|\cdot|_{m,D}$ denote the norms and semi norms. The $L^2$ inner product is then denoted as $(\cdot,\cdot)_D$.


In this section, we develop a unified analysis framework for VEMs. 
We will first establish VEM quintuplets, then present some general hypotheses they should satisfy and show these hypotheses can yield optimal convergence. 
For simplicity's sake, we shall employ the notation $\lesssim$ and $\gtrsim$ representing $\le C$ and $\ge C$ where $C$ is a generic constant independent of element shape and size. 
In addition, the notation $\simeq$ denotes equivalence where the hidden constant $C$ has the same property. 



\subsection{Abstract setup of VEM}
Mimicking the Ciarlet's finite element triplets \cite{Ciarlet1975LecturesOT}, 
we introduce quintuplets for the description of basic ingredients of a VEM. Given an element $K$, we define $(K,\mathcal{B}_h(\partial K), \mathcal{V}_h(K), \mathcal{W}_h(K),\mathcal{D}_K)$ where the components are explained as below
\begin{itemize}
\item $\mathcal{B}_h(\partial K)\subseteq H^1(\partial K)$ is a finite-dimensional space defined on $\partial K$ called the \textit{trace space};
\item $\mathcal{V}_h(K)$ is a finite-dimensional virtual element space defined as
\begin{equation}
\label{lifting}
\mathcal{V}_h(K) = \{ \mathscr{L}_Kv_h: v_h\in \mathcal{B}_h(\partial K) \} \subset H^1(K),
\end{equation}
where $\mathscr{L}_K:\mathcal{B}_h(\partial K)\rightarrow H^1(K)$ is a lifting operator extending boundary functions to the element interior;
\item $\mathcal{W}_h(K) \subset H^1(K)$ is a computable finite-dimensional space onto which $\mathcal{V}_h(K)$ is projected;
\item $\mathcal{D}_K=\{L_1,{L}_2,...,{L}_{N_K}\}$ is a set of linear forms on $\mathcal{B}_h(\partial K)$ such that
\begin{equation}
\label{DoF_def}
v_h \rightarrow (L_1(v_h), L_2(v_h), ..., L_{N_K}(v_h))
\end{equation}
is bijective, which describes the DoFs with $N_K$ being its number.
\end{itemize}
The global space is defined as
\begin{equation}
\label{global_space}
\mathcal{V}_h = \{ v_h|_{\partial \Omega} = 0 :~ v_h|_K \in \mathcal{V}_h(K) , ~ \forall K\in \mathcal{T}_h \} \subseteq H^1_0(\Omega).
\end{equation}

With the preparation above, for some suitable $u$ with sufficient regularity admitting pointwise evaluation, 
we define the interpolation $I_{\partial K} u\in \mathcal{B}_h(\partial K)$ such that $L_j(I_{\partial K} u) = L_j(u), ~ j =1,2,..., N_K$. 
Then, the interpolation on $K$ is defined as
\begin{equation}
\label{bound_interp_K}
I_K u : = \mathscr{L}_K I_{\partial K} u.
\end{equation}
Through this work, we shall use $u_I$ to denote the global interpolation.

Notice that $\beta$ in some cases may not be computed exactly where the error may be caused by surface geometry, quadrature, etc.
To pursue the completeness of the proposed framework, we introduce its approximation denoted as $\beta_h$.
As functions in $\mathcal{V}_h(K)$ are not computable, 
we need a projection operator $\Pi_K: H^1(K) \rightarrow \mathcal{W}_h(K)$ defined as
\begin{equation}
\label{Pi_proj}
( \beta_h \nabla \Pi_K v_h, \nabla w_h)_{K} = ( \beta_h \nabla v_h, \nabla w_h)_{K} , ~~~~ \forall w_h \in \mathcal{W}_h(K),
\end{equation}
where $\int_{\partial K} \Pi_K v_h \dd s = \int_{\partial K}  v_h \dd s$ is imposed for uniqueness. 
As $\forall w_h\in \mathcal{W}_h(K)$ is explicitly known, the projection in \eqref{Pi_proj} is computable:
\begin{equation}
\begin{split}
\label{Pi_proj_compute}
( \beta_h \nabla \Pi_K v_h, \nabla w_h)_{K} & = (  \nabla v_h, \beta_h \nabla w_h)_{K} \\
& = -  (   v_h, \text{div}( \beta_h \nabla w_h) )_{K} + (\beta_h \nabla w_h\cdot\bfn, v_h)_{0,\partial K},
\end{split}
\end{equation}
provided that $v_h$ is known on $\partial K$. 
In the present work, $\text{div}( \beta_h \nabla w_h)$ is actually $0$, i.e., $\beta_h \nabla \mathcal{W}_h$ is a divergence-free space making the first term vanished.
For simplicity, we shall use $\Pi$ to denote the global projection, i.e., $\Pi = \Pi_K$, on each $K$. 

Then, the local discrete bilinear form is defined as $a_K(\cdot,\cdot): H^1(K)\times H^1(K) \rightarrow \mathbb{R}$ where
\begin{equation}
\label{vem_scheme}
    a_K(u_h, v_h):=(\beta_h \nabla \Pi_K u_h, \nabla \Pi_K v_h)_{K}
    +S_K(u_h-\Pi_K u_h, v_h-\Pi_K v_h),
\end{equation}
where $S_K(\cdot,\cdot)$ is another semi-positive symmetric bilinear form, called the \textit{stabilization}, to make $a_K$ stable, as $( \nabla \Pi_K u_h, \nabla \Pi_K v_h)_{0,K}$ itself may not be stable. 
See the Hypotheses in the next subsection and the precise definition in Section \ref{subsec:stab}. 
Furthermore, the global bilinear form is defined as
\begin{equation}
\label{vem_scheme_glob}
a_h(u_h, v_h):=\sum_{K\in \mathcal{T}_h} a_K(u_h, v_h),
\end{equation} 
which should lead to a reasonably good approximation to $a(u_h, v_h)$, $\forall u_h, v_h \in H^1_0(\Omega)$ in \eqref{model} in certain sense.
Now, the virtual scheme is to find $u_h\in \mathcal{V}_h$ such that 
\begin{equation}
\label{eq_VEM}
    a_h(u_h, v_h) =\sum_{K\in \mathcal{T}_h}(f, \Pi_K v_h)_{L^2(K)}, ~~~~~ \forall v_h \in \mathcal{V}_h.
\end{equation}





\begin{remark}
\label{rem_discont}
For standard VEMs, the typical choice of $\mathcal{W}_h(K)$ is $\mathcal{P}_1(K)$. 
In Section \ref{sec:IVEM}, we will discuss an immersed virtual element space to handle discontinuous coefficients on unfitted meshes, 
where $\mathcal{W}_h(K)$ is a non-polynomial space.
\end{remark}


\subsection{General hypotheses}
\label{sec:assump}
Now, let us discuss general hypotheses on $\mathcal{B}_h(\partial K)$, $\mathcal{W}_h(K)$ and $S_K$. 
The bilinear form $a_K(\cdot,\cdot)$ leads to the following quantities:
\begin{equation}
\label{energy_norm}
\vertiii{v_h}^2_K = a_K(v_h,v_h), ~~~ \text{and} ~~~ \vertiii{v_h}^2_h = a_h(v_h,v_h).
\end{equation}


\begin{itemize}
\item[\textbf{(H1)}] \label{asp:H1} (Consistency) $\mathcal{P}_0( K)\subset \mathcal{W}_h( K)$, $\mathcal{P}_0(\partial K)\subset \mathcal{B}_h(\partial K)$ and $\mathcal{P}_0( K)\subset \mathcal{V}_h(K)$. The lifting operator $\mathscr{L}_K$ preserves constants, i.e., $\mathscr{L}_K:\mathcal{P}_0(\partial K) \rightarrow \mathcal{P}_0( K)$. 
  \item[\textbf{(H2)}]\label{asp:H2} (Stability) The stabilization $S_K(\cdot,\cdot)$ is non-negative and leads to a norm $\|\cdot\|_{S_K}$ 
on $\widetilde{\mathcal{B}}^0_h(\partial K): = \{v_h \in \mathcal{B}_h(\partial K) \oplus \text{Tr}_{\partial K} ~ \mathcal{W}_h(K) : (v_h,1)_{\partial K} = 0\}$ such that
\begin{equation}
\label{SK_equiv_1}
\| \cdot \|_{0,\partial K} \lesssim h^{1/2}_K\| \cdot \|_{S_K} ~~~ \text{in} ~ \widetilde{\mathcal{B}}^0_h(\partial K),
\end{equation}
where $\text{Tr}_{\partial K} $ denotes the trace operator on $\partial K$.
  \item[\textbf{(H3)}]\label{asp:H3} (The approximation capabilities of $\mathcal{V}_h(K)$) 
There holds
\begin{equation}
\label{approxi_VK_Pi}
\sum_{K\in \mathcal{T}_h} \vertiii{ u - u_I }^2_K \lesssim h^2 \| u \|^2_{2,\Omega}.
\end{equation}
\item[\textbf{(H4)}] \label{asp:H4} (The approximation capabilities of $\mathcal{W}_h(\partial K)$) There hold
\begin{subequations}
\label{Pi_approx}
\begin{align}
    &   \sum_{K\in \mathcal{T}_h}  \| \nabla (u - \Pi_K u) \|^2_{0,K} \lesssim h^2 \|u\|^2_{2,\Omega} ,   \label{Pi_approx_eq1} \\
    &   \sum_{K\in \mathcal{T}_h}  h_K\| \nabla(u - \Pi_K u) \|^2_{0,\partial K}  + \| u - \Pi_K u \|^2_{S_K}  \lesssim h^2 \|u\|_{2,\Omega} .   \label{Pi_approx_eq2} 
\end{align}
\end{subequations}
\item[\textbf{(H5)}] \label{asp:H5} (The extra hypothesis for $L^2$ estimates) There holds
\begin{equation}
\label{L2_assump}
  \sum_{K\in \mathcal{T}_h}  \| u - \Pi_K u \|^2_{0,K} + h_K \| u -  u_I \|^2_{0,\partial K} \lesssim h^4 \|u\|^2_{2,\Omega} .
\end{equation}

\item[\textbf{(H6)}] \label{asp:H6} (The approximation of $\beta_h$) There holds $\| \beta_h \|_{\infty,\Omega} \le \| \beta \|_{\infty,Omega}$, and
\begin{equation}
\label{betah_approx}
 \sum_{K\in\mathcal{T}_h}   \| \beta \nabla u - \beta_h \nabla u \|^2_{0,K} + h_K \| \beta \nabla u\cdot\bfn -  \beta_h \nabla u \cdot\bfn \|^2_{0,\partial K} \lesssim h^2 \|u\|^2_{2,\Omega} .
\end{equation}

\end{itemize}

\begin{remark}
~\\
\begin{itemize}
\item In the VEM framework, $\mathscr{L}_K$ is typically defined though local problems/PDEs, for instance
\begin{equation}
\label{lift1}
\nabla \cdot(\beta \nabla \mathscr{L}_K v_h )= 0, ~~~ \text{in} ~ K,~~~ v_h|_{\partial K} = b_h, ~~~ \text{on} ~ \partial K ~~~ \forall b_h \in \mathcal{B}_h(\partial K).
\end{equation} 
In this work, we see that this is not essential. The analysis is applicable to any lifting operators satisfying $\mathscr{L}_K:\mathcal{P}_0(\partial K) \rightarrow \mathcal{P}_0( K)$.

\item \hyperref[asp:H1]{(H1)} is referred to as consistence because of the following observation: 
if $\mathcal{P}_0(K)\subset \mathcal{V}_h(K)$, and taking $v_h=c\in\mathcal{P}_0(K)$ and $w_h = \Pi_K v_h$ in \eqref{Pi_proj_compute}, 
then there holds $\nabla\Pi_K v_h = \mathbf{ 0}$ and thus $\mathcal{P}_0(K)\subset \mathcal{W}_h(K)$. 
As $\text{Tr}_{\partial K}\mathcal{V}_h(K) = \mathcal{B}_h(\partial K)$, there holds $\mathcal{P}_0(\partial K)\subset \mathcal{B}_h(\partial K)$. 
In addition, the uniqueness condition $\int_{\partial K} \Pi_K v_h \dd s = \int_{\partial K}  v_h \dd s$ implies $\Pi_K$ preserves constants. 
We shall see in Lemma \ref{lem_norm} below that the constant space bases the coercivity.

\item For classical VEMs on isotropic meshes where $\mathcal{W}_h(K)$ is a polynomial space, 
usually one only needs to estimate $\|u-u_I\|_{j,K}$ and $\|u- \Pi_K u \|_{j,K}$ which can imply all the inequalities above through trace inequalities and Poincare-type inequalities \cite{Brenner;Sung:2018Virtual}.
For anisotropic elements, this implication may not be true, and thus we need to perform more delicate analysis here. 

\item \hyperref[asp:H2]{(H2)} generally holds on isotropic meshes even for functions in $H^1(\partial K)$ with the zero average thanks to the Poincar\'e-type inequality \cite[Section 2.4]{Brenner;Sung:2018Virtual}. 
In Section \ref{sec:boundary}, we shall see that the geometric restrictions can be much relaxed if we only require it to hold for a discrete space. 

\item Not that $\mathcal{W}_h(K)\subset \mathcal{V}_h(K)$ is not required in the present framework. 
But, in most situations, this is indeed true which implies $\text{Tr}_{\partial K} ~ \mathcal{W}_h(K) \subset \mathcal{B}_h(\partial K)$. 
In this case, \hyperref[asp:H2]{(H2)} is just needed for $\mathcal{B}^0_h(\partial K): = \{v_h \in \mathcal{B}_h(\partial K) : (v_h,1)_{\partial K} = 0\}$.

\item Estimation of the energy norms only requires \hyperref[asp:H1]{(H1)}-\hyperref[asp:H4]{(H4)} and \hyperref[asp:H6]{(H6)}, while the $L^2$ norm additionally requires \hyperref[asp:H5]{(H5)}. 
 In \hyperref[asp:H6]{(H6)}, we do not directly assume the estimate of $|\beta - \beta_h|$, as $\beta$ may have singularity. \eqref{betah_approx} is more relaxed.

\end{itemize}
\end{remark}

\begin{lemma}
\label{lem_norm}
Under Hypotheses \hyperref[asp:H1]{(H1)} and \hyperref[asp:H2]{(H2)}, $\vertiii{\cdot}_h$ defines a norm on $\mathcal{V}_h$.
\end{lemma}
\begin{proof}
\hyperref[asp:H1]{(H1)} implicitly implies $\mathcal{P}_0(K)\subset \mathcal{V}_h(K)$.
Given any $v_h\in\mathcal{V}_h$, assume $\vertiii{v_h}_h = 0$. Then, $\| \nabla \Pi_K v_h \|_{0,K}=0$ implies $\Pi_K v_h\in \mathcal{P}_0(K)$. As $(v_h - \Pi_K v_h)|_{\partial K} \in \mathcal{B}^0_h(\partial K)$, by \hyperref[asp:H2]{(H2)}, we have $v_h|_{\partial K} = \Pi_K v_h|_{\partial K}\in\mathcal{P}_0(\partial K)$. 
Using the preserving property of $\mathscr{L}_K$, we have $v_h\in\mathcal{P}_0(K)$ which finishes the proof by the continuity.
\end{proof}


\subsection{A unified analysis framework}

With the hypotheses above, we are able to show the optimal error estimates through a unified procedure. 
We point out that most of the techniques for the energy norm estimation have been established in \cite{Cao;Chen:2018AnisotropicNC}, and some details will be omitted here to avoid redundancy. 
But special attention must be paid to that we only use the hypotheses above without introducing any shape regularity assumptions.  
Meanwhile, the analysis for the $L^2$ norm in the literature \cite{2013BeiraodeVeigaBrezziCangiani,Brenner;Sung:2018Virtual} heavily relies on the estimate for $\|u-u_I\|_{0,K}$ 
which is not available in this work due to the anisotropic meshes. 
So the proposed $L^2$ analysis below is new.



Note that the equation $-\nabla\cdot(\beta \nabla u) \in L^2(\Omega)$ implicitly implies that $u \in H^1(\Omega)$ and $\beta \nabla u \in \bfH(\text{div};\Omega)$. In the forthcoming discussion, we consider the regularity assumption $u\in H^2_{0}(\beta;\Omega)$ where
\begin{equation}
  \label{beta_space_1}
H^2_{0}(\beta;\Omega) = \{ u\in H^1_0(\Omega) \cap H^2(\Omega), ~ \beta \nabla u \in \bfH(\ddiv;\Omega) \}. 
\end{equation}
Note that $H^2_{0}(\beta;\Omega) = H^1_0(\Omega) \cap H^2(\Omega)$ for smooth $\beta$. 
But, we here chose to keep $\beta$, as this space will be slightly modified to adapt to singular $\beta$ in Section \ref{sec:IVEM}, see \eqref{beta_space_2}. 
Nevertheless, for singular $\beta$, the results in this section are still applicable. 
We begin with the following lemma.
\begin{lemma}
\label{lem_est_partialK}
Let $u\in H^2_0(\beta;\Omega)$. Under Hypotheses \hyperref[asp:H1]{(H1)}, \hyperref[asp:H4]{(H4)} and \hyperref[asp:H6]{(H6)}, there holds
\begin{equation}
\label{lem_est_partialK_eq0}
\sum_{K\in\mathcal{T}_h} (\beta \nabla u\cdot \bfn, \Pi_K v)_{\partial K} \lesssim h^{1/2} \| u \|_{2,\Omega} \left( \sum_{K\in\mathcal{T}_h} \| v -\Pi_K v \|^2_{0,\partial K} \right)^{1/2} , ~~~~ \forall v\in H^1(\Omega).
\end{equation}
\end{lemma}
\begin{proof}
We first notice the following identity
\begin{equation}
\label{lem_est_partialK_eq1}
(\beta_h \nabla \Pi_K u\cdot\bfn, \Pi_K v - v)_{\partial K} = (\beta_h \nabla \Pi_K u , \nabla( \Pi_K v - v)  )_{K}  = 0.
\end{equation}
Then, as $v$ is continuous across faces, inserting $v$ and using \eqref{lem_est_partialK_eq1} yields
\begin{equation}
\begin{split}
\label{lem_est_partialK_eq2}
&(\beta \nabla u\cdot \mathbf{ n}, \Pi_K v - v)_{\partial K} = (\beta_h( \nabla u -  \nabla \Pi_K u)\cdot \mathbf{ n} , \Pi_K v - v)_{\partial K} \\
& + ( (\beta-\beta_h) \nabla u\cdot \mathbf{ n}  , \Pi_K v - v)_{\partial K} \\
& \lesssim \left( \| \beta_h ( \nabla u -  \nabla \Pi_K u)\cdot \mathbf{ n} \|_{0,\partial K} + \| (\beta-\beta_h) \nabla u\cdot \mathbf{ n} \|_{0,\partial K } \right) \| \Pi_K v_h - v_h \|_{0,\partial K}.
\end{split}
\end{equation}
The estimate of $\| \beta_h ( \nabla u -  \nabla \Pi_K u)\cdot \mathbf{ n} \|_{0,\partial K}$ is given by \eqref{Pi_approx_eq2} in Hypothesss \hyperref[asp:H4]{(H4)}, while the estimate of $ \| (\beta-\beta_h) \nabla u\cdot \mathbf{ n} \|_{0,\partial K }$ follows from Hypothesss \hyperref[asp:H6]{(H6)}.
\end{proof}


\subsubsection{The energy norm estimate}
We consider the error decomposition:
\begin{equation}
\label{err_decomp}
\xi_h = u - u_I, ~~~~ \text{and} ~~~~ \eta_h = u_I - u_h,
\end{equation}
where $u_h$ is the VEM solution corresponding to the exact solution $u$.


\begin{theorem}[The energy norm]
\label{lem_err_eqn}
Let $u\in H^2_0(\beta;\Omega)$. Under Hypotheses \hyperref[asp:H1]{(H1)}-\hyperref[asp:H4]{(H4)} and \hyperref[asp:H6]{(H6)}, there holds
\begin{equation}
\label{lem_err_eqn_eq0}
\begin{aligned}
\vertiii{ \eta_h}_h \lesssim h \| u \|_{2,\Omega}.
\end{aligned}
\end{equation}
\end{theorem}
\begin{proof}
Applying integration by parts to the equation $-\nabla \cdot (\beta \nabla u) = f$ in $\Omega$, we obtain
\begin{equation}
\label{lem_err_eqn_1}
\begin{aligned}
 & a_h(\eta_h,v_h) = a_h(u_h-u_I, v_h) = a_h(u_h, v_h)-a_h(u_I,v_h) \\
=& \sum_{K\in \mathcal{T}_h}[ \underbrace{(\beta \nabla u, \nabla \Pi_K v_h)_K - (\beta_h \nabla \Pi_K u_I, \Pi_K u_h) }_{(I)} - \underbrace{ (\beta \nabla u\cdot \mathbf{ n}, \Pi_K v_h)_{\partial K} }_{(II)}] \\
& -\underbrace{ S_K(u_I-\Pi_K u_I, v_h-\Pi_K v_h) }_{(III)}.
\end{aligned}
\end{equation}
For $(I)$ in \eqref{lem_err_eqn_1}, we have
\begin{equation*}
\begin{split}
\label{lem_err_eqn_2}
(I) 
=  (\beta_h (\nabla u - \nabla  \Pi_K u), \nabla \Pi_K v_h)_K  + ((\beta-\beta_h) \nabla u, \nabla \Pi_K v_h)_K + (\beta_h \nabla \Pi_K (u-u_I), \nabla \Pi_K v_h)_K,
\end{split}
\end{equation*}
of which the estimates follow from Hypotheses \hyperref[asp:H4]{(H4)}, \hyperref[asp:H6]{(H6)} and \hyperref[asp:H3]{(H3)}, respectively.
For $(II)$, by Lemma \ref{lem_est_partialK}, we only need to estimate $\| v_h-\Pi_K v_h \|_{0, \partial K}$. Hypotheses \hyperref[asp:H2]{(H2)} immediately yields $\| v_h-\Pi_K v_h \|_{0,\partial K} \lesssim h^{1/2}_K  \vertiii{v_h}_h$.
As for $(III)$, we note that $(III) \le \| u_I - \Pi_K u_I \|_{S_K} \| v_h - \Pi_K v_h \|_{S_K}$,
and thus it remains to estimate $\| u_I - \Pi_K u_I \|_{S_K}$. With the triangular inequality, we have
\begin{equation}
\label{lem_err_eqn_9}
\| u_I - \Pi_K u_I \|_{S_K} \le \| u - \Pi_K u \|_{S_K} + \| u - u_I - \Pi_K (u-u_I) \|_{S_K} 
\end{equation}
of which the estimates follow from \eqref{Pi_approx_eq2} and \eqref{approxi_VK_Pi} , respectively. 
\end{proof}

Now, we present the following theorem.
\begin{theorem}
\label{thm_energy_est}
Let $u\in H^2_0(\beta;\Omega)$. Under Hypotheses \hyperref[asp:H1]{(H1)}-\hyperref[asp:H4]{(H4)} and \hyperref[asp:H6]{(H6)}, there holds
\begin{equation}
\label{thm_energy_est_eq0}
\vertiii{u-u_h}_h \lesssim h \| u \|_{2,\Omega}.
\end{equation}
\end{theorem}
\begin{proof}
The estimates of $\vertiii{\xi_h}_h$ and $\vertiii{\eta_h}_h$ follow from \eqref{approxi_VK_Pi} and Theorem \ref{lem_err_eqn}, respectively.
\end{proof}


\subsubsection{The $L^2$ norm estimate}
The $L^2$ norm estimation under anisotropic elements is more difficult. 
We begin with following two corollaries from the energy norm estimate.

\begin{corollary}
\label{cor_uI}
Let $u\in H^2_0(\beta;\Omega)$. Under Hypotheses \hyperref[asp:H1]{(H1)}-\hyperref[asp:H4]{(H4)} and \hyperref[asp:H6]{(H6)}, there holds
\begin{equation}
\label{cor_uI_eq0}
\sum_{K\in\mathcal{T}_h} h_K \| u_I - \Pi_K u_I \|^2_{0,\partial K} \lesssim \sum_{K\in\mathcal{T}_h} \| u_I - \Pi_K u_I \|^2_{S_K} \lesssim h^2 \| u \|_{2,\Omega} .
\end{equation}
\end{corollary}
\begin{proof}
For \eqref{cor_uI_eq0}, the left inequality directly follows from \eqref{SK_equiv_1}. For the right one, notice that 
\begin{equation}
\begin{split}
\label{cor_uI_eq1}
\| u_I - \Pi_K u_I\|_{S_K} & \le \| (u - u_I) - \Pi_K (u - u_I) \|_{S_K} + \| u - \Pi_K u \|_{S_K} 
\end{split}
\end{equation}
of which the estimates follows from Theorem \ref{thm_energy_est} and \eqref{Pi_approx_eq2}.
\end{proof}

\begin{corollary}
\label{cor_energy}
Let $u\in H^2_0(\beta;\Omega)$. Under Hypotheses \hyperref[asp:H1]{(H1)}-\hyperref[asp:H4]{(H4)} and \hyperref[asp:H6]{(H6)}, there holds
\begin{subequations}
\begin{align}
& \sum_{K\in\mathcal{T}_h} \| \nabla( u - \Pi_K u_h ) \|_{0,K} \lesssim h \| u \|_{2,\Omega},  \label{cor_energy_eq1}\\
 &  \sum_{K\in\mathcal{T}_h}\| u_h - \Pi_K u_h \|_{S_K} \lesssim h \| u \|_{2,\Omega}.  \label{cor_energy_eq0}
\end{align}
\end{subequations}
\end{corollary}
\begin{proof}
\eqref{cor_energy_eq1} follows from inserting $u$ into $\| \nabla\Pi_K(u-u_h)\|_{0,K}$. The estimate for \eqref{cor_energy_eq0} is the same as Corollary \ref{cor_uI}
\end{proof}

We are ready to estimate the solution errors under the $L^2$ norm.
\begin{theorem}
\label{thm_u}
Let $u\in H^2_0(\beta;\Omega)$. Under Hypotheses \hyperref[asp:H1]{(H1)}-\hyperref[asp:H6]{(H6)}, there holds
\begin{equation}
\label{thm_u_eq0}
\| u - \Pi u_h \|_{0,\Omega} \lesssim h^2 \| u \|_{2,\Omega}.
\end{equation}
\end{theorem}
\begin{proof}
Let $z\in H^2_0(\beta;\Omega)$ be the solution to $-\nabla\cdot(\beta\nabla z) = \Pi(u-u_h)$. Testing this equation by $\Pi(u-u_h)$ and applying integration by parts, we have
\begin{equation}
\label{thm_u_eq1}
\| \Pi(u-u_h) \|^2_{0,\Omega} = \sum_{K\in\mathcal{T}_h} \underbrace{ ( \beta \nabla z, \nabla\Pi_K(u-u_h) )_K }_{(I)} - \underbrace{ ( \beta \nabla z\cdot\bfn,  \Pi_K(u-u_h)  )_{\partial K} }_{(II)}.
\end{equation}
For $(I)$, we note $( \beta_h \nabla \Pi_K z , \nabla (u-\Pi_K u_h) )_{ K} = ( \beta_h \nabla z , \nabla \Pi_K(u- u_h) )_{ K}$ by the projection property. Then, we have
\begin{equation}
\begin{split}
\label{thm_u_eq2}
(I) 
&= \underbrace{ ( (\beta- \beta_h) \nabla z, \nabla\Pi_K(u-u_h) )_K }_{(Ia)} + \underbrace{ ( \beta_h \nabla ( \Pi_K z - z ), \nabla (u-\Pi_K u_h) )_{ K} }_{(Ib)} \\
& + \underbrace{ ( \beta_h \nabla   z, \nabla (u-\Pi_K  u_h) )_{ K} }_{(Ic)}.
\end{split}
\end{equation}
$(Ia)$ follow from Hypothesis \hyperref[asp:H6]{(H6)}, and $(Ib)$ follows from \eqref{cor_energy_eq1} and \eqref{Pi_approx_eq1} with $\|\beta_h\|_{\infty}\le \|\beta \|_{\infty}$ in Hypothesis \hyperref[asp:H6]{(H6)}.
As for $(Ic)$, by the projection property, we have
\begin{equation}
\label{thm_u_eq5}
(Ic) = ( (\beta_h - \beta) \nabla z, \nabla u )_{ K} + \underbrace{ ( \beta \nabla z, \nabla u )_{ K} }_{\dagger_1} - \underbrace{ ( \beta_h \nabla \Pi_K z, \nabla \Pi_K u_h )_{ K}  }_{\dagger_2},
\end{equation}
where the estimate of the first term still follows from Hypothesis \hyperref[asp:H6]{(H6)}. We proceed to estimate the remaining two terms. 
We notice the following identity by inserting $ \nabla \Pi_Kz_I$ and $ \nabla \Pi_K u_h$:
\begin{equation}
\begin{split}
\label{thm_u_eq6_1}
\dagger_2 &= ( \beta_h \nabla \Pi_K u_h, \nabla \Pi_K z_I )_{ K} +  ( \beta_h (\nabla \Pi_K u_h - \nabla u)), \nabla \Pi_K(z- z_I) )_{ K} \\
& + ( \beta \nabla  u, \nabla \Pi_K(z- z_I) )_{ K}  + ( (\beta_h - \beta) \nabla  u, \nabla \Pi_K(z- z_I) )_{ K}.
 \end{split}
\end{equation}
Then, by the projection property again, applying integration by parts to $\dagger_1$ and $(\beta \nabla u, \nabla \Pi_K( z-z_I) )$, and using the scheme \eqref{eq_VEM} for $( \beta_h \nabla \Pi_K u_h, \nabla \Pi_K z_I )_{ K}$, we arrive at
\begin{equation*}
\begin{split}
\label{thm_u_eq6_2}
\dagger_1 - \dagger_2& = (f, z)_K - (f,\Pi_Kz_I) + S_K(u_h - \Pi_Ku_h, z_I - \Pi_Kz_I) -  ( \beta_h \nabla \Pi_K (u_h-u), \nabla \Pi_K(z- z_I) )_{ K}  \\ 
&- ( f,  \Pi_K(z- z_I) )_{ K} - (\beta \nabla u\cdot\bfn, \Pi_K(z-z_I))_{\partial K}  - ( (\beta_h-\beta) \nabla u, \nabla \Pi_K(z- z_I) )_{ K}  \\
& = (f, z - \Pi_K z)_K + S_K(u_h - \Pi_Ku_h, z_I - \Pi_Kz_I)  -  ( \beta_h \nabla \Pi_K (u_h-u), \nabla \Pi_K(z- z_I) )_{ K} \\
& -  (\beta \nabla u\cdot\bfn, \Pi_K(z-z_I))_{\partial K}  - ( (\beta_h-\beta) \nabla u, \nabla \Pi_K(z- z_I) )_{ K}  
\end{split}
\end{equation*}
where the terms in the right-hand side after being summed over all the elements are denoted as $\ddag_1$,..., $\ddag_5$, respectively. The following estimates are immediate:
\begin{subequations}
\begin{align*}
\ddag_1 &\lesssim h^2 \| f \|_{0,\Omega} \| z \|_{2,\Omega}, &\text{(by \eqref{L2_assump})} \\
 \ddag_2 &\le \sum_{K\in\mathcal{T}_h} \| u_h - \Pi_K u_h \|_{S_K}  \| z_I - \Pi_K z_I \|_{S_K}   \lesssim h^2 \| u \|_{2,\Omega} \| z \|_{2,\Omega}, &\text{(by \eqref{cor_energy_eq0} and \eqref{cor_uI_eq0})} \nonumber \\
 \ddag_3 & \lesssim  \sum_{K\in\mathcal{T}_h}  \| \nabla \Pi_K(u-u_h) \|_{0,K}  \| \nabla \Pi_K(z- z_I) \|_{0,K}  \lesssim h^{2} \| u \|_{2,\Omega} \| z \|_{2,\Omega}, & \text{(by \eqref{thm_energy_est_eq0} and \eqref{approxi_VK_Pi})}  \nonumber \\
 \ddag_4 & \lesssim h^{1/2} \| u \|_{2,\Omega} \left( \sum_{K\in\mathcal{T}_h} \| (z - z_I) - \Pi_K (z-z_I) \|^2_{0, \partial K}  \right)^{1/2} \lesssim h^{2} \| u \|_{2,\Omega} \| z \|_{2,\Omega}, & \text{(by \eqref{L2_assump} and \eqref{cor_uI_eq0})} \nonumber \\
 \ddag_5 & \lesssim h \|  u \|_{2, \Omega} \left( \sum_{K\in\mathcal{T}_h} \| \nabla \Pi_K(z- z_I) \|^2_{0,K} \right)^{1/2} \lesssim h^{2} \| u \|_{2,\Omega} \| z \|_{2,\Omega} . & \text{(by \eqref{approxi_VK_Pi})} 
\end{align*}
\end{subequations}
Putting these estimates into \eqref{thm_u_eq5} leads to the estimate for $(Ic)$, which is combined with $(Ia)$ and $(Ib)$ to conclude the estimate for $(I)$. Next, the estimate of $(II)$ is similar to $\ddag_3$ above. $(I)$ and $(II)$ lead to the estimate of $\Pi(u-u_h)$ by the elliptic regularity $\|z \|_{2,\Omega}\lesssim \| \Pi(u-u_h) \|_{0,\Omega}$. It finishes the proof by applying triangular inequality to $u - \Pi_K u_h = (u - \Pi_Ku) + (\Pi_Ku_h - \Pi_K u)$.
\end{proof}

\section{The boundary space}
\label{sec:boundary}

One feature of the proposed method is a boundary triangulation that enables us to overcome the difficulty arising from anisotropic element shapes. We first make the following assumption:
\begin{itemize}
  \item[(\textbf{A1})] \label{asp:A1} For each element $K$, the number of edges and faces is uniformly bounded. Each of its face admits a triangulation satisfying the 2D maximum angle condition in which the edges are connected by vertices only in $\mathcal{E}_K$.
\end{itemize}  

We shall denote $\mathcal{T}_h(\partial K)$ by the surface mesh, let the collection of all the vertices and edges be $\mathcal{N}_h(\partial K)$ and $\mathcal{E}_h(\partial K)$, respectively.
With Assumption \hyperref[asp:A1]{(A1)}, we define
\begin{equation}
\label{BhK_fitted}
\mathcal{B}_h(\partial K) = \{ v_h \in L^2(\partial K)~:~ v_h|_T \in \mathcal{P}_1(T), ~ \forall T\in \mathcal{T}_h(\partial K) \}.
\end{equation}


\begin{figure}[h]
  \centering
  \begin{minipage}{.4\textwidth}
  \centering
  \includegraphics[width=2.3in]{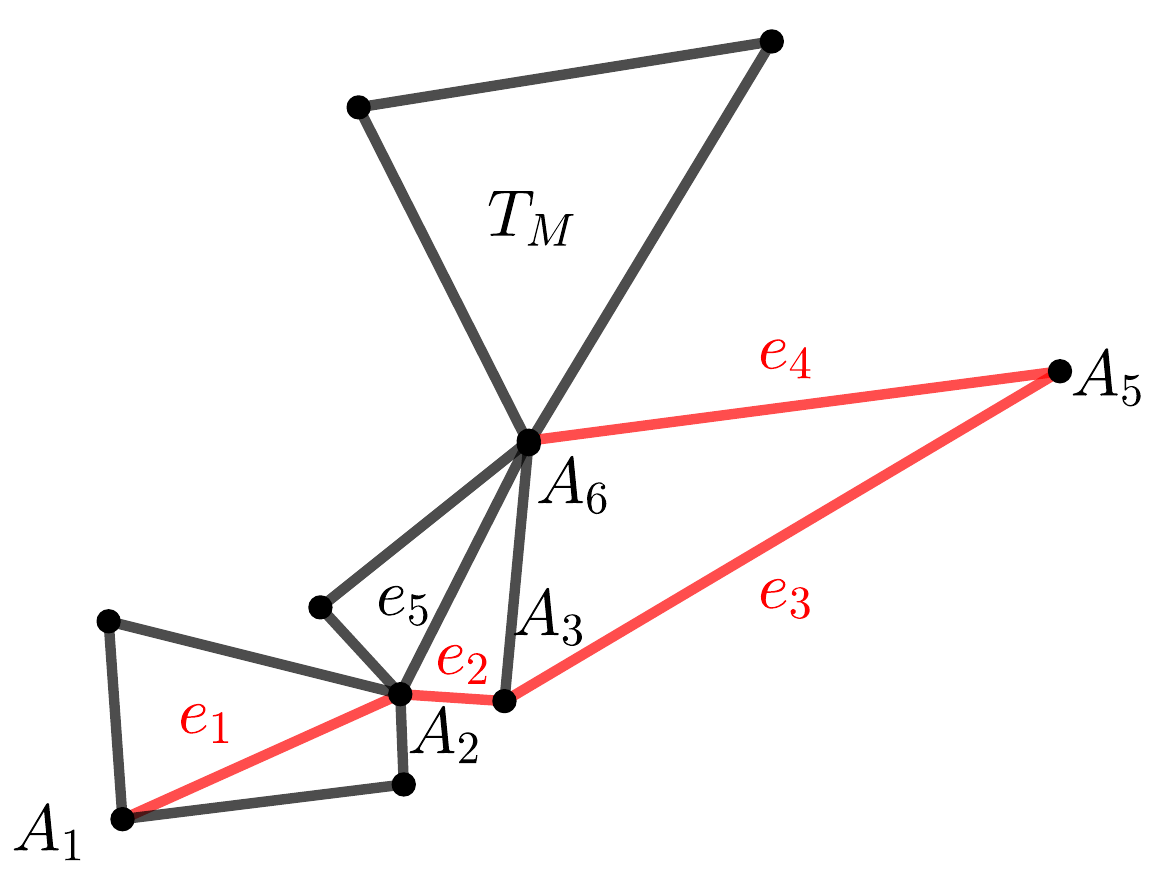}
  \caption{Illustration of Assumption \hyperref[asp:A2]{(A2)}: $e_5 = A_2A_6$ is not allowed in a path as the two neighborhood elements may shrink to this edge. Then, $e_2$, $e_3$ and $e_4$ are needed to connect $A_2$ and $A_6$.}
  \label{fig:path}
  \end{minipage}
  ~~~~
  \begin{minipage}{0.55\textwidth}
   \includegraphics[width=3in]{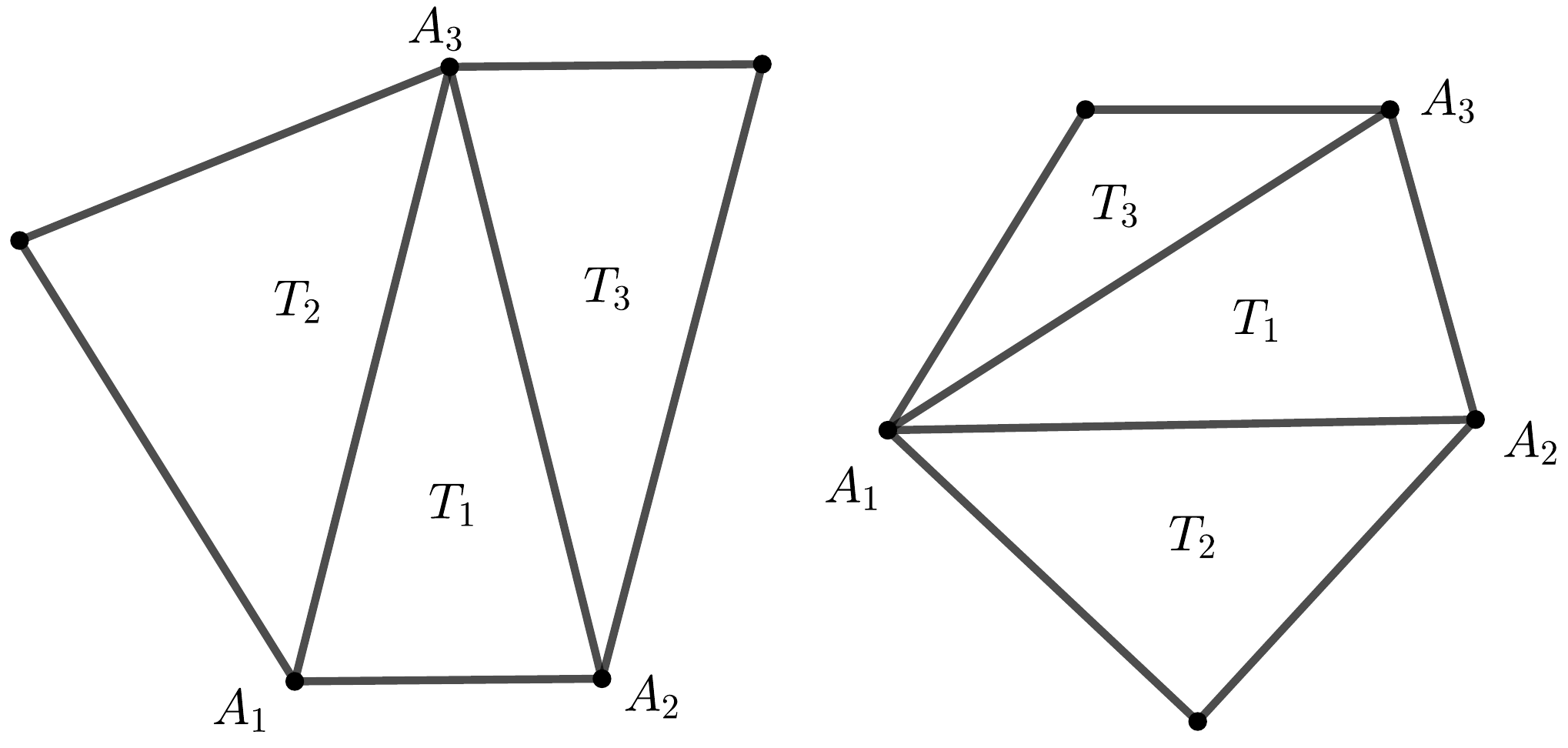}
  \caption{Configuration of anisotropic elements. $T_1$ has the minimum angle at $A_3$ (left) or at $A_1$ (right). Note that these triangles may not be coplanar.}
  \label{fig:anisotrop_triangles}
  \end{minipage}
\end{figure}

We recall the following projection estimate which will be frequently used in this work.
Given a domain $D$ and the non-negative integers $m$, $k$, let be $\mathrm{P}^k_D$ the $L^2$ projection form $H^{m+1}(D)$ to $\mathcal{P}_k(D)$. 
\begin{lemma}[\cite{1999Rudiger}]
\label{lem_proj}
Let $D$ be a domain and let $k\le m$. Then, for every $u\in H^{m+1}(D)$
\begin{equation}
\label{lem_proj_eq0}
\| u -  \mathrm{P}^m_D u \|_{k,D} \leqslant c_{m,k} h^{m+1-k}_D |u|_{m+1,\conv(D)}.
\end{equation}
\end{lemma}

\subsection{A discrete Poincar\'e type inequality}
In this subsection, we discuss a discrete Poincar\'e type inequality which is basically Hypothesis \hyperref[asp:H2]{(H2)}. 
If a polyhedron $K$ is star convex to a ball of the radius $\mathcal{O}(h_K)$, the standard Poincar\'e inequality for $H^1(\partial K)$ holds \cite[Section 2.3]{Brenner;Sung:2018Virtual}. 
But the polyhedra considered in this work are much more irregular.

The maximum angle condition itself is not sufficient for the Poincar\'e inequality, and we introduce a \textit{path condition}. 
A \textit{path} between two vertices is defined as a collection of edges connecting these two vertices. Given each triangle $T$, 
denote $\theta_m(T)$ and $\theta_M(T)$ as the minimum and maximum angles of $T$. 
The following assumption is made:
\begin{itemize}
  \item[(\textbf{A2})] \label{asp:A2} Let $T_M\in\mathcal{T}_h(\partial K)$ have the maximum area. 
  For each vertex $\bfz\in \mathcal{N}_h(\partial K)$, there exists a vertex $\bfz'$ of $T_M$ and a path from $\bfz$ to $\bfz'$ such that, for each edge $e$ in this path, 
  one of its opposite angles $\theta_e$ satisfies $\theta_e \le (1+\epsilon) \theta_m(T)$ with $T$ being the element containing the angle $\theta_e$.
\end{itemize}

\begin{remark}
Note that Assumption \hyperref[asp:A2]{(A2)} does not pose any restrictions on the minimum angle which could be still arbitrarily small. 
Roughly speaking, for an edge in a path, the two neighborhood elements cannot shrink to this edge.
We use Figure \ref{fig:path} for illustration.
\end{remark}

Verifying Assumption \hyperref[asp:A2]{(A2)} is not easy sometimes, as it requires a global check overall all possible paths, which could be expensive if there are many triangles. 
So we introduce an alternative assumption which is much easier to verify, as it only requires local information. 

\begin{itemize}
  \item[(\textbf{A2}')] \label{asp:A2plus} Every triangle $T$ shares at least one edge with another triangle (including $T$ itself) which satisfies the minimum angle condition and whose size bounded below by $\mathcal{O}(h_T)$.
\end{itemize} 

\begin{lemma}
\label{lem_A2plus}
Assumptions \hyperref[asp:A1]{(A1)} and \hyperref[asp:A2plus]{(A2')} together imply \hyperref[asp:A2]{(A2)} with $\epsilon=\theta_M/\arcsin (\rho \sin(\theta_m) \sin(\theta_M))$, 
where $\theta_M$ is the maximum angle in Assumption \hyperref[asp:A1]{(A1)}, 
$\theta_m$ is the minimum angle in Assumption \hyperref[asp:A2]{(A2)}, 
and $\rho \le 1$ is such that the size of the triangle in Assumption \hyperref[asp:A2]{(A2)} greater than $\rho h_T$.
\end{lemma}
\begin{proof}
See Appendix \ref{append:lem_A2plus}.
\end{proof}

\begin{remark}
\label{practical_assump}
In most of the situations, Assumption  \hyperref[asp:A2plus]{(A2')} is sufficient, for example those elements which may shrink to a face, e.g., the middle and right plots in Figure \ref{fig:cub}. 
But for some extreme case that elements may shrink to edges, such as the left plot in Figure \ref{fig:cub}, we still need to use Assumption \hyperref[asp:A2]{(A2)}, as Assumption \hyperref[asp:A2plus]{(A2')} does not hold there. 
In fact, in this case, it is possible that every triangle only has the bounded maximum angle.
\end{remark}

Now, we are ready to establish the Poincar\'e inequality and begin with the following lemma.
 
\begin{lemma}
\label{lem_oppo_angle}
Suppose a triangle $T$ has the maximum angle $\theta_M(T)$. 
Let $e$ be one of its edges with the opposite angle $\theta_e$ and the ending points $\bfx_1$ and $\bfx_2$. 
Suppose $\theta_e\le (1+\epsilon) \theta_m(T)$. Then, there holds
\begin{equation}
\label{lem_oppo_angle_eq0}
| w_h(\bfx_2) - w_h(\bfx_1) | \le \kappa \| \nabla w_h \|_{0,T} ~~~~ \forall w_h \in \mathcal{P}_1(T),
\end{equation}
\end{lemma}
where $\kappa = \sqrt{2}/\sin\left( \frac{\pi-\theta_M(T)}{2+\epsilon } \right)$.
\begin{proof}
If $\theta_e$ itself is the minimum angle, then, both the edges opposite to $\bfx_1$ and $\bfx_2$ should have the length greater than $|e|$.
Thus, we have $|T|\ge \sin(\theta_M(T))|e|^2/2$, and thus
\begin{equation}
\label{lem_oppo_angle_eq1}
| w_h(\bfx_2) - w_h(\bfx_1) | = \abs{ \int_{\bfx_1}^{\bfx_2} \partial_{\bft_e} w_h \dd s } \le |e| \| \nabla w_h \| \le \sqrt{2/\sin(\theta_M)} \| \nabla w_h \|_{0,T}.
\end{equation}
If $\theta_e$ is not the minimum angle, by the assumption we have $(2+\epsilon)\theta_m(T) + \theta_M(T) \ge \pi$ implying $\theta_m(T) \ge (\pi-\theta_M(T))/(2+\epsilon)$. 
In fact, we can let $\theta_e$ be the maximum angle to maximize the bound below. 
Then, by the sine law, there holds $|T| = |e|^2 \sin(\theta_m(T))\sin( \theta_e)/(2\sin(\theta_M(T)))$, and thus
\begin{equation}
\label{lem_oppo_angle_eq2}
| w_h(\bfx_2) - w_h(\bfx_1) | \le |e| \| \nabla w_h \| \le \sqrt{ 2\sin(\theta_M(T))/( \sin(\theta_m(T))\sin(\theta_e)) } \| \nabla w_h \|_{0,T}.
\end{equation}
As $\theta_e \in [\theta_m(T),\theta_M(T)]$, the desired estimate follows from \eqref{lem_oppo_angle_eq2} with the bound of $\theta_m(T)$.
\end{proof}

\begin{lemma}
\label{lem_disct_poinc}
Given a polyhedral element $K$, 
suppose $\partial K$ has a boundary triangulation $\mathcal{T}_h(\partial P)$ with the maximum angle $\theta_M$, 
and suppose it satisfies the Assumptions \hyperref[asp:A1]{(A1)} and \hyperref[asp:A2]{(A2)}. Then, there holds
\begin{equation}
\label{lem_disct_poinc_eq0}
\| v_h - \mathrm{P}^0_{\partial K} v_h \|_{0,\partial P} \le \sqrt{5} \kappa h_K |N_{\mathcal{T}}|^{1/2}  |v_h|_{1,\partial K}, ~~~~ \forall v_h\in \mathcal{B}_h(\partial K),
\end{equation}
where $\kappa$ inherits from Lemma \ref{lem_oppo_angle}, and $N_{\mathcal{T}}$ is the number of elements in the triangulation.
\end{lemma}
\begin{proof}
To simplify the notation, $\nabla$ in this proof is understood as the surface gradient on $\partial P$. 
Let $T_M$ be the triangle with the maximum area, and we trivially have
\begin{equation}
\label{lem_disct_poinc_eq1}
\| v_h - \mathrm{P}^0_{\partial K} v_h \|_{0,\partial K} \le \| v_h - \mathrm{P}^0_{T_M} v_h \|_{0,\partial K}.
\end{equation}
Let $\tilde{v}_h = v_h - \mathrm{P}^0_{T_M} v_h$, and 
let $\bfz$ be the vertex at which $\tilde{v}_h$ achieves the maximum value on $\partial K$. 
Consider the path from Assumption \hyperref[asp:A2]{(A2)} connecting $\bfz$ and one vertex $\bfz'$ of $T_M$. 
Let the path be formed by $\{e_l\}_{l=1}^L$ with the neighborhood triangles $\{T_l\}_{l=1}^L$ described by Assumption \hyperref[asp:A2]{(A2)}.
Lemma \ref{lem_oppo_angle} implies
\begin{equation}
\begin{split}
\label{lem_disct_poinc_eq2}
\| \tilde{v}_h \|^2_{0,\partial K} \le |\partial K| |\tilde{v}_h(\bfz)|^{2} \le 2 |\partial K| |\tilde{v}_h(\bfz')|^{2} + 2 L  |\partial K| \kappa^2 \| \nabla v_h \|^2_{0,\cup_{l=1}^L T_l}.
\end{split}
\end{equation}
As $\tilde{v}_h$ must vanish at one point in $T_M$, and it is a linear polynomial, 
we simply have $\| \tilde{v} \|_{\infty,\partial T_M} \le h_{T_M}|T_M|^{-1/2} \| \nabla v_h \|_{0,T_M}$. 
Noticing $L \le 3|\mathcal{N}_T|/2$, we derive from \eqref{lem_disct_poinc_eq2} that
\begin{equation}
\label{lem_disct_poinc_eq3}
\| \tilde{v}_h \|^2_{0,\partial K} \le 2|\partial K|/|T_M| h^2 \| \nabla v_h \|^2_{0,T_M} + 3|\mathcal{N}_T|  |\partial K| \kappa^2 \| \nabla v_h \|^2_{0,\partial K}
\le 5|\mathcal{N}_T| h^2_K \kappa^2 \| \nabla v_h \|_{0,\partial K},
\end{equation}
where we have also used $|\partial K|/|T_M| \le |\mathcal{N}_T|$. It finishes the proof.
\end{proof}

\begin{remark}
\label{rem_disct_poinc}
For a 2D polygonal region $D$ with a triangulation $\mathcal{T}_h(D)$, we can show a similar estimate. 
Remarkably, the constants in these estimates are very explicitly specified and independent of the shape. 
It is worthwhile to mention that the constant in \eqref{lem_disct_poinc_eq0} goes to $\infty$ if the triangulation is very fine. 
This property, in fact, agrees with the bound of the classical Poincar\'e inequality for irregular domains, in the sense that the discrete space will approach the $H^1$ space.
\end{remark}



\subsection{Stabilization}
\label{subsec:stab}

With the boundary space in \eqref{BhK_fitted}, we consider the stabilization
\begin{equation}
\label{stab_f}
S_K(v_h,w_h) = h_K \sum_{F\in\mathcal{F}_K} (\nabla_F v_h,\nabla_F w_h)_{0,F}
\end{equation}
which is computable since functions in $\mathcal{B}_h(\partial K)$ are known given the boundary triangulation.

\begin{lemma}
\label{lem_stab_f_verify}
Under Assumptions \hyperref[asp:A1]{(A1)} and \hyperref[asp:A2]{(A2)}, \eqref{SK_equiv_1} in Hypothesis \hyperref[asp:H2]{(H2)} holds for $S_K$ in \eqref{stab_f}:
\begin{equation}
\label{lem_stab_f_verify_eq0}
\| \cdot \|_{0,\partial K} \lesssim h_K \|\cdot \|_{S^f_K}~~~ \text{in} ~ \mathcal{B}^0_h(\partial K) .
\end{equation}
\end{lemma}
\begin{proof}
The result immediately follows from Lemma \ref{lem_disct_poinc}.
\end{proof}
We shall only discuss $S_K$ in \eqref{lem_stab_f_verify_eq0} in this work, but other stabilizations can be also considered.

\begin{remark}
\label{rem_stab}
Without explicitly forming the boundary triangulation, we can also consider
\begin{equation}
\label{stab_e}
\widetilde{S}_K(v_h,w_h) = h^{2}_K \sum_{e\in\mathcal{E}_K} (\partial_{\bft_e} v_h ,\partial_{\bft_e}w_h )_{0,e}.
\end{equation}
For $S_K = \widetilde{S}_K$, \eqref{SK_equiv_1} holds under an extra assumption for the boundary triangulation. 
\begin{itemize}
  \item[(\textbf{A1'})] \label{asp:A1E} Any edge $e\in\mathcal{E}_h(\partial K)\backslash \mathcal{E}_K$, i.e., an extra edge connected by some vertices, should either have the length $\mathcal{O}(h_K)$ or is one edge of a polygon whose other edges are from $\mathcal{E}_K$ and have the length $\mathcal{O}(h_e)$.
\end{itemize}  
Under Assumptions \hyperref[asp:A1E]{(A1)} and \hyperref[asp:A1E]{(A1')}, applying Lemma \ref{lem_grad} we are able to show that
\begin{equation}
\label{lem_stab_e_verify_eq0}
h^{-1}_K\| \cdot \|_{0,\partial K} \lesssim \| \cdot \|_{S^f_K}\lesssim \| \cdot \|_{S^e_K}, ~~~~ \text{in} ~ \mathcal{B}_h(\partial K).
\end{equation}
Computing $S^e_K$ is more efficient than $S^f_K$ as it only requires evaluating the function values on original edges in $\mathcal{E}_K$ (not the edges in $\mathcal{E}_h(\partial K)$). 
\end{remark}

 \section{Application I: elements with non-shrinking inscribed balls}
 \label{sec:fitted_1}

In this section, we analyze the proposed method on one type of anisotropic meshes: elements are allowed to merely contain \textit{but are not necessarily star convex to} non-shrinking balls.

To facilitate a clear presentation, we shall let $\beta=\beta_h=1$, and thus the virtual spaces defined through \eqref{lifting} and \eqref{lift1} become those classical in the literature. 
Let $\mathcal{W}_h(K) = \mathcal{P}_1(K)$, and thus $\nabla\Pi_K\cdot$ is just the standard $L^2$ projection to the constant vector space. 
All these setups are widely employed in the VEM literature. 
It is then trivial that $\mathcal{W}_h(K)$, $\mathscr{L}_K$ and $\mathcal{B}_h(\partial K)$ satisfy Hypothesis \hyperref[asp:H1]{(H1)}. 
Hypothesis \hyperref[asp:H2]{(H2)} has been discussed in Section \ref{subsec:stab}, and Hypothesis \hyperref[asp:H6]{(H6)} is trivial. 
We proceed to examine other Hypotheses in Section \ref{sec:assump}. 
Through this section, we remind readers that the standard interpolation estimates based on the maximum angle condition in \cite{1999AcostaRicardo,1976BabuskaAziz} are not directly applicable as it requires higher regularity assumptions, see Remark \ref{rem_interp_maxangle} below.



We make the following assumption.
\begin{itemize}
    \item[(\textbf{A3})] \label{asp:A3} Each element $K$ contains a ball $B_K$ of the radius $\mathcal{O}(h_K)$. 
    In addition,  there are $K_j$, $j=1,...,r$ such that $\conv(K) \subset \cup_{j=1}^r K_j$ with $r$ uniformly bounded.
\end{itemize}
\begin{remark}
\label{rem_nonshriking}
We highlight that Assumption \hyperref[asp:A3]{(A3)} does not require the star convexity with respect to $B_K$, 
so it is much weaker than the one in \cite{Brenner;Sung:2018Virtual}. 
In addition, it does not require that each face has a supporting height $\mathcal{O}(h_K)$ towards $K$, 
so it is also weaker than \cite{Cao;Chen:2018AnisotropicNC}. See Figure \ref{fig:anisotropic_element} for an example. 
Nevertheless, we point out that $\conv(K)$ is indeed convex with respect to $B_K$ which is fundamental for the analysis in this Section.
\end{remark}


Based on this assumption, we have the following trace inequality only for polynomials.
\begin{lemma}[A trace inequality on anisotropic elements]
\label{lem_traceinequa}
Under Assumption \hyperref[asp:A3]{(A3)}, there holds 
$$
\| v_h \|_{0,\partial K} \lesssim h^{-1/2}_K \| v_h \|_{0,K}, ~~~~ \forall v_h \in \mathcal{P}_1(K).
$$
\end{lemma}
\begin{proof}
Let $B_K$ be the largest inscribed ball of $K$ with the center $O$, and let $\widetilde{B}_K$ be the ball centering at $O$ of the radius $h_K$. 
Clearly, $B_K \subset K\subset \conv(K) \subset \widetilde{B}_K$, and $\widetilde{B}_K$ is a homothetic mapping of $B_K$ of the ratio $\rho_{\widetilde{B}_K}/\rho_{{B}_K}\lesssim 1$ by Assumption \hyperref[asp:A3]{(A3)}. 
(\textbf{G1}) in Lemma \ref{lem_tet_maxangle} and Lemma 2.2 in \cite{2016WangXiaoXu} yield 
$$
\| v_h \|_{0,\partial K} \lesssim h^{-1/2}_K \| v_h \|_{0,\conv(K)} \lesssim h^{-1/2}_K \| v_h \|_{0,\widetilde{B}_K} 
 \lesssim h^{-1/2}_K \| v_h \|_{0,{B}_K}  \lesssim h^{-1/2}_K \| v_h \|_{0,K}.
$$ 
\end{proof}

Next, we estimate the interpolation errors on the boundary $\partial K$. 
\begin{lemma}
\label{lem_interp_1D}
Given an edge $e$, let $I_e$ be the 1D interpolation on $e$. Then, $\forall u\in H^1(e)$
\begin{equation}
\label{lem_interp_1D_eq0}
| I_e u |_{1,e} \le |u|_{1,e}.
\end{equation}
\end{lemma}
\begin{proof}
Let $A_1$ and $A_2$ be the two ending points of $e$. It follows from the H\"older's inequality that
\begin{equation}
\label{lem_interp_1D_eq1}
\| \partial_{\bft_e} I_e u \|^2_{0,e} = |e|^{-1} (u(A_2) - u(A_1))^2 = |e|^{-1} \left( \int_{e} \partial_{\bft_e} u \dd s \right)^2 \le \| \partial_{\bft_e} u \|_{0,e}.
\end{equation}
\end{proof}

\begin{lemma}
\label{lem_uI_face}
Let $u\in H^2(\conv(K))$.
Under Assumptions \hyperref[asp:A1]{(A1)} and \hyperref[asp:A2]{(A2)}, there holds
\begin{subequations}
\label{lem_uI_face_eq0}
\begin{align}
   &  \| u - u_I \|_{0, \partial K} \lesssim h^{3/2}_K \| u \|_{E,2,\conv(K)}  , \label{lem_uI_face_eq01}  \\ 
    &   | u - u_I |_{1,\partial K} \lesssim h^{1/2}_K \| u \|_{E,2,\conv(K)} . \label{lem_uI_face_eq02} 
\end{align}
\end{subequations}
\end{lemma}
\begin{proof}
Given an element $K$ and a triangle $T\in\mathcal{F}_h(\partial K)$,
consider the projection $\mathrm{P}^k_{\conv(K)}$, $k=0,1$. For \eqref{lem_uI_face_eq01}, we have
\begin{equation}
\label{lem_uI_face_eq1}
\| u - u_I \|_{0, T} \le \| u - \mathrm{P}^1_{\conv(K)} u \|_{0, T}  +  \| u_I - \mathrm{P}^1_{\conv(K)} u \|_{0, T} .
\end{equation}
The trace inequality with Assumption \hyperref[asp:A1]{(A1)}, Lemma \ref{lem_proj} and \textbf{(G1)} in Lemma \ref{lem_tet_maxangle} imply
\begin{equation}
\begin{split}
\label{lem_uI_face_eq2}
\| u - \mathrm{P}^1_{\conv(K)} u \|_{0, T} & \lesssim h^{-1/2}_K \| u - \mathrm{P}^1_{\conv(K)} u \|_{0,\conv(K)} \\ 
&+ h^{1/2}_K  | u - \mathrm{P}^1_{\conv(K)} u |_{1,\conv(K)} \lesssim h^{3/2}_K |u|_{2,\conv(K)}.
\end{split}
\end{equation}
As for the second term in \eqref{lem_uI_face_eq1}, noticing that $u_I - \mathrm{P}^1_{\conv(K)} u\in \mathcal{P}_1(T)$, we trivially have 
\begin{equation}
\label{lem_uI_face_eq3}
\| u_I - \mathrm{P}^1_{\conv(K)} u \|_{0, T} \lesssim |T|^{1/2} \abs{ (u_I - \mathrm{P}^1_{\conv(K)} u )(\bfa) },
\end{equation} 
where $\bfa$ is some vertex of $T$.
We consider the shape-regular tetrahedron $T'$ given by (\textbf{G3}) in Lemma \ref{lem_tet_maxangle} that has $\bfa$ as one vertex, and let $I_{T'}$ be the standard Lagrange interpolation on $T'$. Then, by applying the trace inequality and the triangular inequality, we obtain
\begin{equation}
\begin{split}
\label{lem_uI_face_eq4}
 \abs{ (u_I - \mathrm{P}^1_{\conv(K)} u)(\bfa)  }& \lesssim h^{-1}_K \| I_T u - \mathrm{P}^1_{\conv(K)} u \|_{0,T'} \\
 & \lesssim h^{-1}_K \left(  \| I_T u -  u \|_{0,T'} +  \| \mathrm{P}^1_{\conv(K)} u -  u \|_{0,T'}  \right)  \lesssim h_K | u |_{2,\conv(K)}.
 \end{split}
\end{equation}
Putting \eqref{lem_uI_face_eq4} into \eqref{lem_uI_face_eq3} and combining it with \eqref{lem_uI_face_eq2}, we have \eqref{lem_uI_face_eq01}.

Next, we prove \eqref{lem_uI_face_eq02}. The triangular inequality yields
\begin{equation}
\label{lem_uI_face_eq5}
| u - u_I |_{1,T} \le \| \nabla_{\partial K}(u - \mathrm{P}^1_{\conv(K)}) u \|_{0,T} + \| \nabla_{\partial K}(\mathrm{P}^1_{\conv(K)} u - u_I) \|_{0,T}.
\end{equation}
The estimate of the first term on the right-hand side in \eqref{lem_uI_face_eq5} follows from the similar argument to \eqref{lem_uI_face_eq2} with the trace inequality. 
We focus on the second term in \eqref{lem_uI_face_eq5}. By Lemma \ref{lem_grad}, we have 
\begin{equation}
\begin{split}
\label{lem_uI_face_eq6}
\| \nabla_{\partial K} ( u_I - \mathrm{P}^1_{\conv(K)} u ) \|_{0,T} \lesssim & \sum_{e\subseteq\partial T} h^{1/2}_T \| \nabla_{\partial K} (u_I - \mathrm{P}^1_{\conv(K)} u )\cdot \bft_e \|_{0,e} .
 \end{split}
\end{equation}
By (\textbf{G2}) in Lemma \ref{lem_tet_maxangle}, we have shape-regular trapezoid $T'$ and pyramid $T''$ contained in $\conv(K)$ with the size $\mathcal{O}(h_K)$. 
Note that $I_eu = u_I$ on $e$. 
Then, by Lemma \ref{lem_interp_1D} and Lemma 2.1 in \cite{Brenner;Sung:2018Virtual}, we have
\begin{equation}
\begin{split}
\label{lem_uI_face_eq7}
& \| \partial_{\bft_e}( I_e u -  \mathrm{P}^1_{\conv(K)} u ) \|_{0,e}  = \| \partial_{\bft_e}I_e ( u -  \mathrm{P}^1_{\conv(K)} u ) \|_{0,e} \\
 \lesssim & |  u -  \mathrm{P}^1_{\conv(K)} u  |_{1,e}  
 \lesssim h^{-1/2}_K |  u -  \mathrm{P}^1_{\conv(K)} u  |_{1,T'} + |  u  |_{3/2,T'} \\
  \lesssim & h^{-1}_K |  u -  \mathrm{P}^1_{\conv(K)} u  |_{1,T''} +   |  u -  \mathrm{P}^1_{\conv(K)} u  |_{2,T''} + |  u   |_{2,T''}
    \lesssim  |  u   |_{2,\conv(K)}.
\end{split}
\end{equation}
Noticing $\nabla_{\partial K}\cdot \bft_e = \partial_{\bft_e} \cdot$, and putting \eqref{lem_uI_face_eq6} and \eqref{lem_uI_face_eq7} into \eqref{lem_uI_face_eq5}, we finish the proof.
\end{proof}

\begin{remark}
\label{rem_interp_maxangle}
Interpolation estimates on a triangle $T$ with the maximum angle condition generally demand relatively higher regularity \cite{1999Duran,1994Shenk}:
\begin{equation}
\label{rem_interp_maxangle_eq1}
\| \nabla(u- I_T) \|_{0,T} \lesssim h_T \| u \|_{2,T}.
\end{equation} 
But as triangulation is on faces on which $u$ has merely $H^{3/2}$ regularity, \eqref{rem_interp_maxangle_eq1} can not be applied directly to obtain the bound in terms of $\|\cdot\|_{2,\conv(K)}$. 
In fact, $\conv(K)$ in \eqref{lem_uI_face_eq0} can be replaced by any shape-regular region containing $K$. 
If $T$ is a tetrahedron, then the interpolation estimate requires even higher regularity \cite{1999Duran}, i.e., $W^{p,2}$, $p>2$, which cannot be further improved, see the counterexample in \cite{1994Shenk}. 
This property adds more complexity to the anisotropic analysis for 3D shrinking elements, see the discussion in the next section.
\end{remark}

Now, Hypotheses \hyperref[asp:H3]{(H3)}-\hyperref[asp:H5]{(H5)} follow from the above estimates.
\begin{lemma}
\label{lem_assump_verify1}
Under Assumptions \hyperref[asp:A1]{(A1)} and \hyperref[asp:A3]{(A3)}, Hypothesis \hyperref[asp:H3]{(H3)} holds.
\end{lemma}
\begin{proof}
Let $\bfp_h = \nabla \Pi_K( u - u_I) \in [\mathcal{P}_0(K)]^3$. Integration by parts and Lemmas \ref{lem_uI_face} and \ref{lem_traceinequa} lead to
$$
\| \bfp_h \|^2_{0,K} = (\bfp_h\cdot\bfn, u-u_I)_{0,\partial K} \le \| \bfp_h \|_{0,\partial K} \| u-u_I \|_{0,\partial K}  \lesssim \| \bfp_h \|_{0,K}  h_K \| u \|_{2,\conv(K)} .
$$
Cancelling one $\| \bfp_h \|_{0,K}$ yields the estimate of $\| \nabla \Pi_K(u-u_I) \|_{0,K}$.
In addition, we note that
\begin{equation}
\label{assump_verify2}
\| (u - u_I) - \Pi_K(u-u_I) \|_{S_K} \le h^{1/2}_K \| \nabla_{\partial K}(u-u_I) \|_{0,\partial K} + h^{1/2}_K \| \nabla_{\partial K} \Pi_K (u-u_I) \|_{0,\partial K}
\end{equation}
of which the first term directly follows from Lemma \ref{lem_uI_face}, and the second term follows from the estimate of $\| \nabla \Pi_K(u-u_I) \|_{0,K}$ and the trace inequality in Lemma \ref{lem_traceinequa}. 
\end{proof}

\begin{lemma}
\label{lem_assump_verify2}
Under Assumptions \hyperref[asp:A1]{(A1)} and \hyperref[asp:A3]{(A3)}, Hypothesis \hyperref[asp:H4]{(H4)} holds.
\end{lemma}
\begin{proof}
Note that \eqref{Pi_approx_eq1} is simple by inserting $\nabla \Pi_{\conv(K)} u$ and using Lemma \ref{lem_proj}. 
For \eqref{Pi_approx_eq2}, we only need to estimate $\| \nabla(u-\Pi_K u) \|_{0,\partial K}$. 
The triangular inequality yields
\begin{equation}
\label{lem_assump_verify2_eq1}
\| \nabla(u-\Pi_K u) \|_{0,\partial K} \lesssim \| \nabla(u-\Pi_{\conv(K)} u) \|_{0,\partial K} + \| \nabla(\Pi_{\conv(K)} u - \Pi_K u )  \|_{0,\partial K}.
\end{equation}
The estimate of the first term on the right-hand side follows from the trace inequality with (\textbf{G1}) in Lemma \ref{lem_tet_maxangle}, 
while the estimate of the second term follows from Lemmas \ref{lem_traceinequa} and \ref{lem_proj} by inserting $\nabla u$.
\end{proof}

At last, we show \eqref{L2_assump}. 
Note that one cannot directly apply the Poincar\'e inequality here, such as (2.15) in \cite{Brenner;Sung:2018Virtual}, since elements are not shape regular. 
\begin{lemma}
\label{lem_assump_verify3}
Under Assumptions \hyperref[asp:A1]{(A1)}-\hyperref[asp:A3]{(A3)}, Hypothesis \hyperref[asp:H5]{(H5)} holds.
\end{lemma}
\begin{proof}
The second term in \eqref{L2_assump} follows from Lemma \ref{lem_uI_face}. 
We focus on the first term. 
Given any $z\in H^1(K)$, we may write $\Pi_K z = (\mathrm{P}_K^0\nabla z) \cdot(\bfx - \bfx_0) +\mathrm{P}^0_{\partial K} z $ 
where $\bfx_0$ is chosen such that $(\bfx - \bfx_0,1)_{\partial K} = \mathbf{ 0}$. 
We have $\| \bfx - \bfx_0 \| \lesssim h_K$, $\forall\bfx\in K$. Then,  letting $v = \mathrm{P}^1_{\conv(K)} u  - u$, we write
\begin{equation}
\begin{split}
\label{lem_assump_verify3_1}
\| u - \Pi_K u \|_{0,K} & \le \| u - \mathrm{P}^1_{\conv(K)} u \|_{0,K} + \| \mathrm{P}^1_{\conv(K)} u  - \Pi_K u \|_{0,K} \\
& \le \| u - \mathrm{P}^1_{\conv(K)} u \|_{0,K} + \|  (\mathrm{P}_K^0\nabla v) \cdot(\bfx - \bfx_0) \|_{0,K} + \| \mathrm{P}^0_{\partial K} v  \|_{0,K}.
\end{split}
\end{equation}
 The estimates of the first two terms follow from Lemma \ref{lem_proj}. 
For the last term, the isoperimetric inequality $|K|\lesssim |\partial K|^{3/2}$, 
the trace inequality with (\textbf{G1}) in Lemma \ref{lem_tet_maxangle} and Lemma \ref{lem_proj} imply
\begin{equation}
\begin{split}
\label{lem_assump_verify3_2}
\| \mathrm{P}^0_{\partial K} v  \|_{0,K}  & = | K|^{1/2}/|\partial K| (v ,1)_{\partial K}
\le | K|^{1/2}/|\partial K|^{1/2} \| v  \|_{0,\partial K} \\
& \lesssim |\partial K|^{1/4} ( h^{-1/2}_K \| v  \|_{0, \conv(K)} + h^{1/2}_K | v  |_{1, \conv(K)}  ) \lesssim h^2_K \| u \|_{2,\conv(K)}.
\end{split}
\end{equation}
Putting \eqref{lem_assump_verify3_2} into \eqref{lem_assump_verify3_1} finishes the proof.
\end{proof}

As a conclusion, under Assumptions \hyperref[asp:A1]{(A1)}-\hyperref[asp:A3]{(A3)}, Theorems \ref{lem_err_eqn} and \ref{thm_u} give the desired estimate.


\section{Application II: a special class of shrinking elements cut from cuboids}
\label{sec:fitted _2}

In this section, we discuss a special class of elements that are cut from cuboids by an arbitrary plane, 
and the inscribed balls of these elements can be accordingly arbitrarily small, as shown in Figure \ref{fig:cub}.
\begin{itemize}
    \item[(\textbf{A4})]  \label{asp:A4} 
    Elements $K$ are cut from cuboids by plane. 
\end{itemize}
Clearly, these elements do not satisfy Assumption \hyperref[asp:A3]{(A3)} and may even shrink to a flat plane or a segment. 
Denote $R$ by the cuboid with the size $h_R$. For an element $K$ cut from $R$, we have $h_K \le  h_R$.


For the considered special class of elements, by rotation, there are three types of elements highlighted by the red solids in Figure \ref{fig:cub}. By \cite[Proposition 3.2]{2017ChenWeiWen}, all these elements have a boundary triangulation satisfying Assumptions \hyperref[asp:A1]{(A1)} with the maximum angle $144^{\circ}$. In addition, the first case in Figure \ref{fig:cub} satisfies Assumption \hyperref[asp:A2]{(A2)}, while the second and third cases satisfy \hyperref[asp:A2plus]{(A2')}.
According to \cite{2000Pflaum}, generating a triangulation for such elements satisfying the maximum angle condition is possible, but the resulting triangulation may not satisfy the delaunay property on element faces; namely the triangulation is non-conforming for elements to elements.
Thus, only requiring the boundary triangulation will greatly simplify the computation \cite{2017ChenWeiWen}. 

We follow the setup of the previous section by setting $\beta=\beta_h=1$ and $\mathcal{W}_h(K)=\mathcal{P}_1(K)$. 
Similarly, we still only need to examine Hypotheses \hyperref[asp:H3]{(H3)}-\hyperref[asp:H5]{(H5)}.
Note that one cannot apply the trace inequalities including \eqref{lem_traceinequa}, say on the face $A_1D_1D_3A_5$ of the left plot in Figure \ref{fig:cub}, towards the element, 
which is the main difficulty. We will see that the involved analysis techniques are completely different from the prevision section and those in the literature.


To avoid redundancy, in the following discussion, we focus on the highlighted element in the left plot in Figure \ref{fig:cub} 
which is a triangular prism and may shrink to the segment $A_1A_5$ and thus, from the perspective of analysis, is the most challenging case.  

\begin{figure}[h]
  \centering
  \includegraphics[width=1.5in]{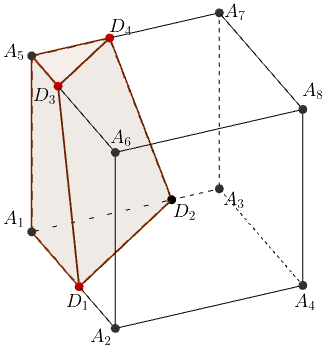}
  \includegraphics[width=1.5in]{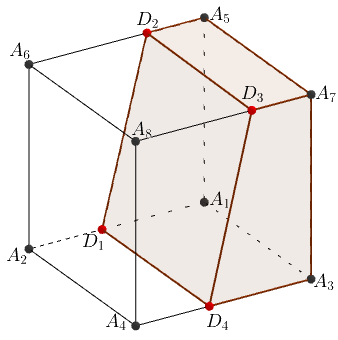}
  \includegraphics[width=1.5in]{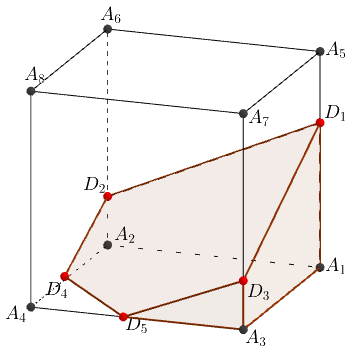}
  \caption{Cases 1-3 from left to right: elements cut from cuboids, highlighted as the red solids, may not contain non-shrinking inscribed balls. 
  The left one may shrink to the segment $A_1A_5$. 
  The middle one and the right one may shrink to the plane $A_1A_3A_7A_5$ and $A_1A_2A_3A_4$. 
  They all satisfy Assumption \hyperref[asp:A1]{(A1)}. Case 1 satisfies Assumption \hyperref[asp:A2]{(A2)}, while Cases 1 and 3 satisfy Assumption \hyperref[asp:A2plus]{(A2')}.
  }
  \label{fig:cub}
\end{figure}


We first present the following lemma which will be frequently used.
\begin{lemma}
\label{lem_edge_est}
Let $u\in H^2(R)$ with $K\subseteq R$. Given each edge $e$ of $K$ with the unit directional vector $\bft_e$, there holds that
\begin{equation}
\label{lem_edge_est_eq0}
\abs{ \int_K \nabla(u-u_I)\cdot\bft_e  \dd \bfx } \lesssim h^{1/2}_K |F|^{-1/2} |K| \| u \|_{2,R},
\end{equation}
where $F$ is any face containing $e$.
\end{lemma}
\begin{proof}
We consider the edge $e$ as $D_1D_3$ and $D_1A_1$ which are representative cases in that the first one does not shrink and the second one may shrink. Denote their directional vectors, respectively, by $\bft_1$ and $\bft_2$. 
In fact, for $e=D_1D_3$, we can take $F$ to be either one of the two faces $A_1D_1D_3A_5$ and $D_1D_2D_4D_3$ by symmetry. For $e=D_1A_1$, as $|A_1D_1D_3A_5|>|A_1D_1D_2|$, estimating \eqref{lem_edge_est_eq0} for $F=A_1D_1D_3A_5$ is enough.
Hence, without loss of generality, we let $F$ be the face $A_1D_1D_3A_5$.
To facilitate a clear presentation, we assume that $D_1$ is the origin, $\bft_2$ is the $x_2$ axis and the face $\triangle A_1D_1D_3$ is on the $x_2x_3$ plane, see Figure \ref{fig:crosssec1} for illustration. 
The rest of the discussion is a bit technical and lengthy, and thus we shall decompose it into several steps. We first consider the case of $D_1D_3$. 

\textit{Step 1. (Rewrite the volume integral as a boundary integral)} 
Now, let us estimate the volume integral in the right-hand side of \eqref{lem_edge_est_eq0} by rewriting it into the integral of cross-sections. We let $F$ be the face $A_1D_1D_3A_5$ and choose the cross-sections parallel to $F$. 
Let $\tilde{x}_1$ be the maximal height of $K$ from the $x_2x_3$ plane. 
Let $\mathbb{T}(x_1)$ be the cross-section at each $x_1$ perpendicular to the $x_1$ axis, and thus it is always parallel to the $A_1D_1D_3A_5$ plane. 
We can then rewrite the integral as
\begin{equation}
\label{lem_assump_verify4_eq2}
\int_K  \nabla (  u - u_I ) \cdot \bft_1 \dd \bfx = \int_{0}^{\tilde{x}_1} \left( \int_{\mathbb{T}(x_1)}  \nabla_{\mathbb{T}} (  u - u_I ) \cdot \bft_1 \dd x_2 \dd x_3 \right) \dd x_1,
\end{equation}
where $\nabla_{\mathbb{T}}$ is the surface gradient within the plane ${\mathbb{T}}(x_1)$. 
For a fixed $\mathbb{T}(x_1)$, without loss of generality, assume it is a quadrilateral denoted by $B_1B_2B_3B_4$, as shown in Figure \ref{fig:crosssec1}, 
where the edge $B_1B_2$ is parallel to $\bft_1$, i.e., $D_1D_3$. 
Note that the cross-sections may be triangular, but it does not affect the analysis below.
As $\bft_1$ is parallel to the $x_2x_3$ plane, it can be written as $\bft_1 = [t_{1,2},t_{1,3}]$ within this plane by dropping ``$0$'' in the first coordinate.
Let $s_h\in \mathcal{P}_1(\mathbb{T}(x_1))$ be sought such that $\rot(s_h) = \bft_1= [ t_{1,2}, t_{1,3} ]$, where $\rot = [\partial_{x_3}, -\partial_{x_2} ]$ is the 2D rotation operator. 
Specifically, we can write 
\begin{equation}
\label{lem_assump_verify4_eq2_1}
s_h(x_2,x_3) =  (x_3 -\bar{x}_3) t_{1,2} - (x_2 -\bar{x}_2) t_{1,3}, 
\end{equation}
where $(\bar{x}_2,\bar{x}_3)$ is the center of $B_1B_2B_3B_4$. With integration by parts, we obtain
\begin{equation}
\begin{split}
\label{lem_assump_verify4_eq5}
\abs{ \int_{\mathbb{T}(x_1)}  \nabla_{\mathbb{T}} ( u - u_I ) \cdot \bft_1 \dd x_2 \dd x_3 } & = \abs{ \int_{\partial\mathbb{T}(x_1)}  \left( \nabla_{\mathbb{T}} (  u - u_I ) \cdot \bft \right) s_h \dd s } \\
 & \le \underbrace{ \|\nabla ( u - u_I)\cdot\bft \|_{0,\partial \mathbb{T}(x_1)} }_{(I)}  \underbrace{\| s_h \|_{0,\partial \mathbb{T}(x_1)}}_{(II)}.
 \end{split}
\end{equation}

\textit{Step 2. (Estimation of the terms $(I)$ and $(II)$)} For $(I)$, there holds
\begin{equation}
\label{lem_assump_verify4_eq6}
\| \nabla ( u - u_I) \cdot \bft \|_{0,\partial \mathbb{T}(x_1)} \le \| (\nabla u - \mathrm{P}^0_R \nabla u )\cdot \bft \|_{0,\partial \mathbb{T}(x_1)}  + \|  (  \mathrm{P}^0_R \nabla u - \nabla u_I)\cdot \bft \|_{0,\partial \mathbb{T}(x_1)} .
\end{equation}
For the first term above, let $e$ be one edge of $ \mathbb{T}(x_1) = B_1B_2B_3B_4$, and we can always find a triangle, denoted by $f$, such that $e$ is one edge of $f$, 
and $f$ is shape regular in the sense of being star convex to a circle of the radius $\mathcal{O}(h_K)$. 
We can then find a shape regular polyhedron inside $R$ that has $f$ has its face. 
For example, if $e=B_1B_2$, we let $f = \triangle B_1B_2A_8$.
Then, using the similar argument to \eqref{lem_uI_face_eq7} with Lemma 2.1 in \cite{Brenner;Sung:2018Virtual} and the trace inequality, we have
\begin{equation}
\begin{split}
\label{lem_assump_verify4_eq7}
\| (\nabla u - \mathrm{P}^0_R \nabla u )\cdot \bft \|_{0,e} & \lesssim h^{-1/2}_K \| \nabla u - \mathrm{P}^0_R \nabla u \|_{0,f} + | \nabla u - \mathrm{P}^0_R \nabla u |_{1/2,f}  \lesssim  \| u \|_{2,R}.
\end{split}
\end{equation}
For the second term on the right-hand side of \eqref{lem_assump_verify4_eq6}, based on the face triangulation assumption, $\partial \mathbb{T}(x_1)$ is covered by several triangles,
and thus can be decomposed into a collection of edges.
Without loss of generality, we consider one edge $e = \triangle D_1D_2D_3 \cap \partial \mathbb{T}(x_1)$.
Due to the maximum angle condition, there exist two edges $e_1$ and $e_2$ of $\triangle D_1D_2D_3$ whose angle is bounded below and above.
By Lemma \ref{lem_edge_max} and the similar argument to \eqref{lem_uI_face_eq7}, we have
\begin{equation}
\label{lem_assump_verify4_eq8}
\|  (  \mathrm{P}^0_R \nabla u - \nabla u_I)\cdot \bft_e \|_{0,e} \lesssim \sum_{i=1,2} \|  (  \mathrm{P}^0_R \nabla u - \nabla u_I)\cdot \bft_{e_i} \|_{0,e_i} \lesssim \| u \|_{2,R}.
\end{equation}
Putting \eqref{lem_assump_verify4_eq7} and \eqref{lem_assump_verify4_eq8} into \eqref{lem_assump_verify4_eq6} and summing it over $e$, we arrive at
\begin{equation}
\label{lem_assump_verify4_eq8_1}
\| \nabla ( u - u_I) \cdot \bft \|_{0,\partial \mathbb{T}(x_1)} \lesssim \| u \|_{2,R}.
\end{equation}

As for $(II)$ in \eqref{lem_assump_verify4_eq5}, 
noticing that $\mathbb{T}(x_1)= B_1B_2B_3B_4$ is convex, $\forall (x_2,x_3)\in \partial \mathbb{T}(x_1)$, there holds 
\begin{equation}
\label{lem_assump_verify4_eq3}
| s_h (x_2,x_3)| = | (x_2 -\bar{x}_2, x_3 -\bar{x}_3)\cdot \bfn_{B_1B_2} |  \le l_{B_1B_2},
\end{equation}
where $\bfn_{B_1B_2}$ is the normal vector to the edge $B_1B_2$ within the plane $\mathbb{T}(x_1)$, and $l_{B_1B_2}$ is the height of the edge $B_1B_2$ towards $\mathbb{T}(x_1)$, see Figure \ref{fig:quad} for illustration. 
Using geometry and \eqref{lem_assump_verify4_eq3}, we have
\begin{equation}
\label{lem_assump_verify4_eq4}
\| s_h \|_{0,\partial \mathbb{T}(x_1) } \lesssim | \partial \mathbb{T}(x_1) |^{1/2} l_{B_1B_2} \lesssim h^{1/2}_K l_{B_1B_2} \lesssim | \mathbb{T}(x_1)| h^{-1/2}_K.
\end{equation}
We point out that the key of \eqref{lem_assump_verify4_eq4} is to get the bound in terms of the possibly shrinking term, $l_{B_1B_2}$, in the second inequality, 
Now, putting \eqref{lem_assump_verify4_eq8_1} and \eqref{lem_assump_verify4_eq4} into \eqref{lem_assump_verify4_eq5} and \eqref{lem_assump_verify4_eq2}, we have
\begin{equation}
\label{lem_assump_verify4_eq9}
\int_K  \nabla (  u - u_I ) \cdot \bft_1 \dd \bfx \lesssim \int_0^{\tilde{x}_1} h^{-1/2}_K \| u \|_{2,R} | \mathbb{T}(x_1)| \dd x_1 \lesssim h^{-1/2}_K |K| \| u \|_{2,R},
\end{equation}
which yields \eqref{lem_edge_est_eq0} as $|F|\lesssim h^2_K$.

\textit{In the next stage, we consider the case of $D_1A_1$ and $\bft_2$, and the proof needs to be slightly modified.}

In this case, we consider the cross-sections parallel to $ A_1D_1D_2$ which are thus perpendicular to the $x_3$ axis. 
So each such cross-section can be described by $\mathbb{T}(x_3)$ denoted as $\triangle C_1C_2C_3$. 
In this case, $C_1C_2$ is always parallel to $\bft_2$. See the right plot in Figure \ref{fig:crosssec1} for illustration.
Then, the volume integral becomes 
\begin{equation}
\label{lem_assump_verify4_eq9_1}
\int_K \nabla(u-u_I)\cdot \bft_2 \dd \bfx = \int_0^{\tilde{x}_3} \left( \int_{\mathbb{T}(x_3)} \nabla(u-u_I)\cdot \bft_2 \dd x_1 \dd x_2 \right) \dd x_3.
\end{equation}
Under the same coordinate system, as $\bft_2$ is parallel to the $x_1x_2$ plane, it can be written as $\bft_2 = (t_{2,1},t_{2,2})$ by dropping $0$ in the third coordinate.

Without loss of generality, we assume $h_{D_1A_1}\ge h_{D_3A_5}$. Note that $C_1C_2$ is always perpendicular to $C_1C_3$. We need to consider two cases.
If $h_{A_1D_1} \ge h_{A_1D_2}$, i.e., $h_{C_1C_2}\ge h_{C_1C_3}$,  then we let $\tilde{s}_h(x_1,x_2) =  (x_2 -\bar{x}_2) t_{2,1} - (x_1 -\bar{x}_1) t_{2,2}$ be
constructed similar to \eqref{lem_assump_verify4_eq2_1} which still leads to $|\tilde{s}_h(x_1,x_2)|\le h_{C_1C_3}$, $\forall(x_1,x_2)\in \partial \mathbb{T}(x_3)$, see the middle plot in Figure \ref{fig:quad} for illustration. 
The argument above is applicable. 

For $h_{A_1D_1} \le h_{A_1D_2}$, i.e., $h_{C_1C_2}\le h_{C_1C_3}$, one can see that \eqref{lem_assump_verify4_eq4} does not hold anymore as it is possible $h_{C_1C_2} / h_{C_1C_3} \rightarrow 0$.
Then, the argument is modified by changing the integration-by-parts strategy. 
Taking $\tilde{s}_h = ((x_1,x_2)-(\bar{x}_1,\bar{x}_2))\cdot\bft_2$, i.e., $\nabla_{\mathbb{T}} \tilde{s}_h = \bft_2$, 
and then we conclude $| \tilde{s}_h(x_2,x_3)|\le h_{C_1C_2}$. As $h_{C_1C_2}\le h_{C_1C_3}$, we have
\begin{equation}
\label{lem_assump_verify4_eq5_1}
\| \tilde{s}_h \|_{0, \partial \mathbb{T}(x_3)  } \lesssim | \partial \mathbb{T}(x_3)|^{1/2} h_{C_1C_2}  \lesssim | \mathbb{T}(x_3)|^{1/2} h^{1/2}_{C_1C_2}.
\end{equation}
Inserting $ \nabla_{\mathbb{T}} \mathrm{P}^1_R u$ into the integral on the cross-section in \eqref{lem_assump_verify4_eq9_1}, we obtain
\begin{equation}
\begin{split}
\label{lem_assump_verify4_eq12}
& \abs{ \int_{\mathbb{T}(x_3)}  \nabla_{\mathbb{T}} ( u - u_I ) \cdot \bft_2 \dd x_1 \dd x_2 }  \le \underbrace{ \abs{ \int_{\mathbb{T}(x_3)}  ( \nabla_{\mathbb{T}} \mathrm{P}^1_R u - \nabla_{\mathbb{T}}  u_I ) \cdot \bft_2  \dd x_1 \dd x_2 } }_{(III)} \\
 + & \underbrace{ \abs{ \int_{\partial\mathbb{T}(x_3)}  \left( ( \nabla_{\mathbb{T}}   u - \nabla_{\mathbb{T}} \mathrm{P}^1_R u ) \cdot \bfn_{\partial\mathbb{T}} \right) \tilde{s}_h \dd s } }_{(IV)} + \underbrace{ \abs{ \int_{\mathbb{T}(x_3)}    \Delta_{\mathbb{T}} ( \mathrm{P}^1_R u - u )   \tilde{s}_h \dd x_1 \dd x_2 } }_{(V)},\\
 \end{split}
\end{equation}
where the terms $(IV)$ and $(V)$ come from integration by parts (for $\nabla$ operator, not the $\rot$ operator). 
The estimate of $(IV)$ is similar to \eqref{lem_assump_verify4_eq7} and \eqref{lem_assump_verify4_eq3}:
we first conclude $\| ( \nabla_{\mathbb{T}}   u - \nabla_{\mathbb{T}} \mathrm{P}^1_R u ) \cdot \bfn_{\partial\mathbb{T}} \|_{\partial \mathbb{T}(x_1)} \lesssim \| u \|_{2,R}$, 
and then using \eqref{lem_assump_verify4_eq5_1} and $\tilde{x}_3 \le h_K$, we obtain
\begin{equation}
\label{lem_assump_verify4_eq13}
|F|^{1/2}/|K|\int_0^{\tilde{x}_3}(IV) \dd x_3 \lesssim (|F|^{1/2}/|K|) |K|^{1/2} h^{1/2}_K h^{1/2}_{A_1D_1}  \| u \|_{2,R} \lesssim h^{1/2}_K  \| u \|_{2,R},
\end{equation}
where, by the assumption of $h_{A_1D_1} \le h_{A_1D_2}$, we have applied $h_{C_1C_2} \le h_{A_1D_1} \le h_{A_1D_2}$ and $|F| h_{A_1D_1} \le |F| h_{A_1D_2} \le |K|$.
For $(V)$, using $|s_h(x_2,x_3)|\le  h_{C_1C_2}$ and $|F| h_{C_1C_3} \le |F| h_{A_1D_2} \le |K|$ again, we have
\begin{equation}
\begin{split}
\label{lem_assump_verify4_eq14}
|F|^{1/2}/|K|\int_0^{\tilde{x}_3}(V) \dd x_3 & \lesssim |F|^{1/2}/|K| h_{C_1C_2} \int_0^{\tilde{x}_3}  \int_{\mathbb{T}(x_3)} \abs{ \Delta_{\mathbb{T}}  u  } \dd x_1 \dd x_2 \dd x_3 \\
& \lesssim |F|^{1/2}/|K| h_{C_1C_3} |K|^{1/2} | u |_{2,K} \lesssim h^{1/2}_K  \| u \|_{2,R} .
\end{split}
\end{equation} 
In addition, for $(III)$, we notice
\begin{equation}
\label{lem_assump_verify4_eq15}
(III) = | \mathbb{T}(x_3) | \abs{ ( \nabla_{\mathbb{T}} \mathrm{P}^1_R u - \nabla_{\mathbb{T}}  u_I ) \cdot \bft_2 } =  | \mathbb{T}(x_3) |^{1/2} h^{1/2}_{C_1C_3} \| ( \nabla_{\mathbb{T}} \mathrm{P}^1_R u - \nabla_{\mathbb{T}}  u_I ) \cdot \bft_2 \|_{0,C_1C_2}.
\end{equation}
Note that the estimation of $\| ( \nabla_{\mathbb{T}} \mathrm{P}^1_R u - \nabla_{\mathbb{T}}  u_I ) \cdot \bft_2 \|_{0,C_1C_2}$ is similar to \eqref{lem_assump_verify4_eq8_1}. 
Thus, by the geometrical inequalities $| \tilde{x}_3 | \le h_K$, $|F| \le h_{A_1D_1} h_K$, $ | \mathbb{T}(x_3) | \le h_{C_1C_2} h_{C_1C_3} \le h_{A_1D_1} h_{A_1D_2}$ and $h_{A_1D_1}h_{A_1D_2}h_K \lesssim |K|$, we arrive at
\begin{equation}
\label{lem_assump_verify4_eq16}
|F|^{1/2}/|K|\int_0^{\tilde{x}_3}(III) \dd x_3 \lesssim ( h^2_{A_1D_1} h^2_{A_1D_2} h_K )^{1/2} h_K /|K|   \| u \|_{2,R} \lesssim h^{1/2}_K \| u \|_{2,R}.
\end{equation}
Combining \eqref{lem_assump_verify4_eq13}, \eqref{lem_assump_verify4_eq14} and \eqref{lem_assump_verify4_eq16} finishes the proof of the case of $\bft_2$.
\end{proof}


\begin{lemma}
\label{lem_cylin_trace}
Let $u\in H^1(R)$ with $K\subseteq R$. Under Assumptions \hyperref[asp:A1]{(A1)} and \hyperref[asp:A4]{(A4)}, $\forall F\in \mathcal{F}_K$,
\begin{equation}
\label{lem_cylin_trace_eq0}
\| u - \mathrm{P}_K^0 u \|_{0,F} \lesssim h^{1/2}_K \| \nabla u \|_{0,R}.
\end{equation}
\end{lemma}
\begin{proof}
When $F= \triangle A_1D_1D_2$ or $\triangle A_5D_3D_4$, the trace inequality immediately yields the desired result.
So we only need to consider the other three faces whose height towards $K$ may shrink.  
Without loss of generality, we let $F$ be the face $D_1D_2D_3D_4$.
By the triangular inequality, we have
\begin{equation}
\label{lem_cylin_trace_eq1}
\| u - \mathrm{P}_K^0 u \|_{0,F}  \le \| u - \mathrm{P}_{F}^0 u \|_{0,F} + \| \mathrm{P}_{F}^0 u - \mathrm{P}_K^0 u \|_{0,F}.
\end{equation}
Let $P_F$ be the pyramid contained in $R$ that has $F$ as its base and $A_8$ as the apex, and it is easy to see that the height of $F$ is $\mathcal{O}(h_K)$.
Then, by the trace inequality, we obtain
\begin{equation}
\begin{split}
\label{lem_cylin_trace_eq2}
\| u - \mathrm{P}_{F}^0 u \|_{0,F} & \lesssim \| u - \mathrm{P}_{R}^0 u \|_{0,F}  \lesssim h^{-1/2}_K \| u - \mathrm{P}_{R}^0 u \|_{0,P_F}  + h^{1/2}_K | u - \mathrm{P}_{R}^0 u |_{1,P_F} \\
& \lesssim h^{-1/2}_K \| u - \mathrm{P}_{R}^0 u \|_{0,R}  + h^{1/2}_K | u  |_{0,R} \lesssim h^{1/2}_K |u|_{1,R}.
\end{split}
\end{equation}
For the second term in \eqref{lem_cylin_trace_eq1}, by the projection property we have
\begin{equation}
\label{lem_cylin_trace_eq3}
\| \mathrm{P}_{F}^0 u - \mathrm{P}_K^0 u \|_{0,F} = |F|^{1/2}/|K| \abs{ \int_K \mathrm{P}_{F}^0 u -u \dd \bfx } \le  |F|^{1/2}/|K| ^{1/2} \| u - \mathrm{P}_{F}^0 u \|_{0,K}.
\end{equation}
Let $l_F$ be the height of $F$ towards $K$.
Note that it is possible $|F|/|K| \rightarrow \infty$ as $K$ shrinks to $F$ i.e., $l_F \rightarrow 0$.
But we will see that the estimate of $\|  u - \mathrm{P}_F^0 u \|_{0,K}$ can compensate the issue.
Without loss of generality, we let $F$ be on the $x_1x_2$ plane with $D_1D_3$ being the $x_1$ axis, see Figure \ref{fig:prism} for illustration, 
and assume the dihedral angle associated with $D_1D_2$ is not greater than $\pi/2$. 
In this case, by the elementary geometry, we know that every face angle and dihedral angle associated with the vertex $D_1$ is uniformly bounded below and above. 
Let $D_1D_3$, $D_1D_2$ and $D_1A_1$ 
have the directional vectors $\bft_1$, $\bft_2$, $\bft_3$. Then, we can construct an affine mapping 
\begin{equation}
\label{lem_cylin_trace_eq3_1}
 \bfx = \mathfrak{F}( \hat{\bfx}) := \left[  \bft_1, ~  \bft_2, ~  \bft_3  \right] \hat{\bfx} := A \hat{\bfx}.
\end{equation}
Note that $\| A^{-1} \|\lesssim |\text{det}(A)|^{-1} = |(\bft_1\times \bft_2)\cdot\bft_3|^{-1}\lesssim 1$ as every angle is bounded below and above, and $\|A\|\lesssim 1$. 
So the mapping \eqref{lem_cylin_trace_eq3_1} maps to $K$ from a triangular prism $\hat{K}= \mathfrak{F}^{-1}(K)$ with the two triangular faces orthogonal to $\hat{F}=   \mathfrak{F}^{-1}(F)$, as shown in Figure \ref{fig:prism}. 
Note that the prism $\hat{K}$ satisfies the condition of Lemma \ref{lem_poin_anis}. 
Let $v= u - \mathrm{P}_F^0 u$ and $\hat{v}(\hat{\bfx}) := v(\mathfrak{F}( \hat{\bfx}) )$, and then we have 
\begin{equation}
\label{lem_cylin_trace_eq4}
\| v \|_{0,K} \lesssim \| \hat{v} \|_{0,\hat{K}} \lesssim l^{1/2}_{\hat{F}} \| \hat{v} \|_{0,\hat{F}} + l_{\hat{F}} \| \nabla \hat{v} \|_{0,\hat{K}}
\lesssim l^{1/2}_F \| u - \mathrm{P}_F^0 u \|_{0,F} + l_F \| \nabla u \|_{0,K}
\end{equation}
where we have used $l_{\hat{F}}\simeq l_F$. 
Substituting \eqref{lem_cylin_trace_eq4} into \eqref{lem_cylin_trace_eq3} yields
\begin{equation}
\begin{split}
\label{lem_cylin_trace_eq5}
\| \mathrm{P}_{F}^0 u - \mathrm{P}_K^0 u \|_{0,F} & \lesssim  (|F| l_F/|K|)^{1/2} \left(  \| u - \mathrm{P}_F^0 u \|_{0,F} + l^{1/2}_F \| \nabla u \|_{0,K}  \right) \lesssim h_K \| \nabla u \|_{0,R},
\end{split}
\end{equation}
where we have used $|F| l_F \le |K|$ and \eqref{lem_cylin_trace_eq2}.
\end{proof}

\begin{figure}[h]
  \centering
  \begin{minipage}{.28\textwidth}
  \centering
  \includegraphics[width=1.4in]{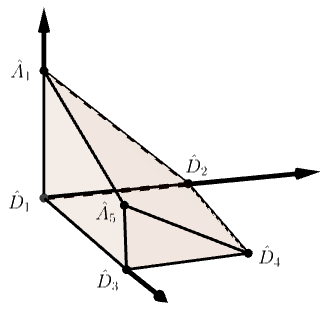}
  \caption{The reference triangular prism $\hat{K}$ mapped to $K$ by $\mathfrak{F}$. The two triangular faces $\hat{D}_1\hat{D}_2\hat{A}_1$ and $\hat{D}_3\hat{D}_4\hat{A}_5$ are perpendicular to the face $\hat{D}_1\hat{D}_3\hat{D}_4\hat{D}_2$}
  \label{fig:prism}
  \end{minipage}
  ~
  \begin{minipage}{0.66\textwidth}
  \centering
   \includegraphics[width=1.7in]{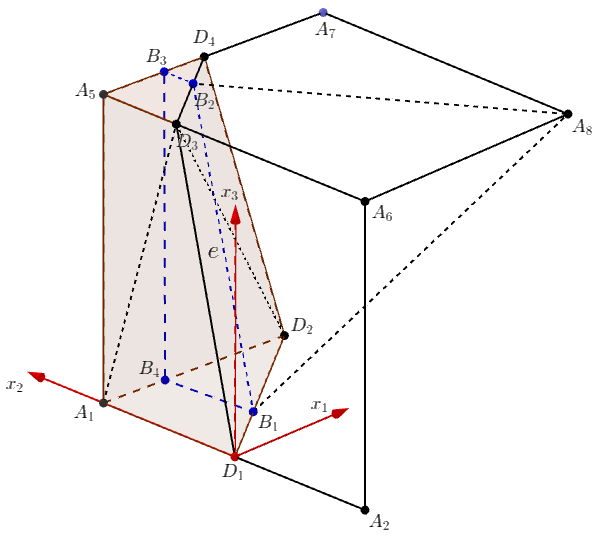}
  \includegraphics[width=1.7in]{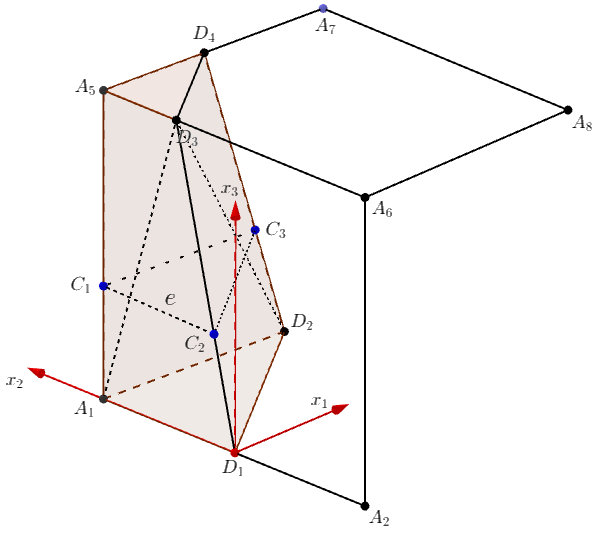}
  \caption{Illustration of the proof of Lemma \ref{lem_edge_est}. 
  Left: the proof for $\bft_1$ being the directional vector of $D_1D_3$. 
  Right: the proof for $\bft_2$ being the directional vector of $D_1A_1$, 
  where the two cases are considered separately: $h_{A_1D_1}\le h_{A_1D_2}$ and $h_{A_1D_1} \ge h_{A_1D_2}$.}
  \label{fig:crosssec1}
  \end{minipage}
\end{figure}

With the preparation above, we are ready to verify Hypotheses \hyperref[asp:H3]{(H3)}-\hyperref[asp:H5]{(H5)}. 

\begin{lemma}
\label{lem_assump_verify4}
Under Assumptions \hyperref[asp:A1]{(A1)} and \hyperref[asp:A4]{(A4)}, Hypothesis \hyperref[asp:H3]{(H3)} holds.
\end{lemma}
\begin{proof}
We first estimate $\|\nabla\Pi_K(u-u_I)\|_{0,K}$. 
Note that there exist three orthogonal edges $e_i$, $i=1,2,3$, of $K$ with the directional vectors $\bft_i$. Using the projection property and Lemma \ref{lem_edge_est}, we have
\begin{equation}
\begin{split}
\label{lem_uI_face2_eq2}
& \| \nabla \Pi_K ( u - u_I )  \|_{0,K}  = |K|^{1/2} \| \nabla \Pi_K ( u - u_I )  \|  
\lesssim |K|^{1/2} \sum_{i=1,2,3} \| \nabla \Pi_K ( u - u_I ) \cdot\bft_i \| \\
 = & \sum_{i=1,2,3} |K|^{-1/2} \int_K \nabla \Pi_K ( u - u_I ) \cdot\bft_i  \dd \bfx 
 =  \sum_{i=1,2,3} |K|^{-1/2} \int_K \nabla  ( u - u_I ) \cdot\bft_i  \dd \bfx \\
 \le & h^{1/2}_K (|K|/|F|)^{1/2} \| u \|_{2,R} \le h_K \| u \|_{2,R} .
\end{split}
\end{equation}

For the stabilization term, using the argument similar to Lemma \ref{lem_uI_face_eq0}, we have
\begin{equation}
\label{lem_uI_face2_eq2_1}
\| u - u_I \|_{0, \partial K} + h_R | u - u_I |_{1,\partial K} \lesssim h^{3/2}_K \| u \|_{2,R}
\end{equation}
Now, let us concentrate on the estimate of $\| \nabla_{\partial K} \Pi_K( u - u_I ) \|_{0,\partial K}$. 
Given each triangular element $F\in\mathcal{T}_h(\partial K)$, if $F$ is the bottom or top triangular face, i.e., $\triangle A_1D_1AD_2$ and $\triangle A_5D_3AD_4$,
the trace inequality immediately yields the desired result. 
The difficult part is the estimate for the three sided quadrilateral faces where the trace inequality is not applicable due to the possibly shrinking height. 
Based on symmetry, we can assume that $F$ is one triangular element of the quadrilateral face, which, without loss of generality, can be taken as $\triangle A_1D_1D_3$. 

By Lemma \ref{lem_grad}, thanks to the maximum angle condition, there exist two edges $e_i$, $i=1,2$, with their tangential directional vectors $\bft_i$ such that
\begin{equation}
\begin{split}
\label{lem_assump_verify4_eq1}
\| \nabla_{F} \Pi_K( u - u_I ) \|_{0,F} & \lesssim \sum_{i=1,2} |F|^{1/2} \| \nabla \Pi_K( u - u_I ) \cdot \bft_i \| \\
& = |F|^{1/2}/|K| \abs{ \sum_{i=1,2} \int_K  \nabla (  u - u_I ) \cdot \bft_i \dd \bfx }.
\end{split}
\end{equation}
For $F=\triangle A_1D_1D_3$, we can take $\bft_1$ and $\bft_2$ as the directional vectors of $D_1D_3$ and $D_1A_1$. 
Then, Lemma \ref{lem_edge_est} finishes the proof.
\end{proof}

\begin{lemma}
\label{lem_assump_verify5}
Under Assumptions \hyperref[asp:A1]{(A1)} and \hyperref[asp:A4]{(A4)}, Hypothesis \hyperref[asp:H4]{(H4)} holds.
\end{lemma}
\begin{proof}
\eqref{Pi_approx_eq1} is trivial by Lemma \ref{lem_proj}, while \eqref{Pi_approx_eq2} immediately follows from Lemma \ref{lem_cylin_trace}.
\end{proof}

\begin{lemma}
\label{lem_assump_verify6}
Under Assumptions \hyperref[asp:A1]{(A1)}, \hyperref[asp:A2]{(A2)} and \hyperref[asp:A4]{(A4)}, Hypothesis \hyperref[asp:H5]{(H5)} holds.
\end{lemma}
\begin{proof}
The argument is the same as Lemma \ref{lem_assump_verify3}, 
where the only difference is to apply Lemma \ref{lem_cylin_trace} to $\|v \|_{0,\partial K}$ in \eqref{lem_assump_verify3_2}. 
\end{proof}

In summary, Assumptions \hyperref[asp:A1]{(A1)}, \hyperref[asp:A2]{(A2)} and \hyperref[asp:A4]{(A4)} imply 
Hypotheses \hyperref[asp:H1]{(H1)}-\hyperref[asp:H6]{(H6)} which further yield the desired optimal estimates.
We refer readers to some numerical results in \cite{2017ChenWeiWen}.
\begin{figure}[h]
  \centering
   \includegraphics[width=5in]{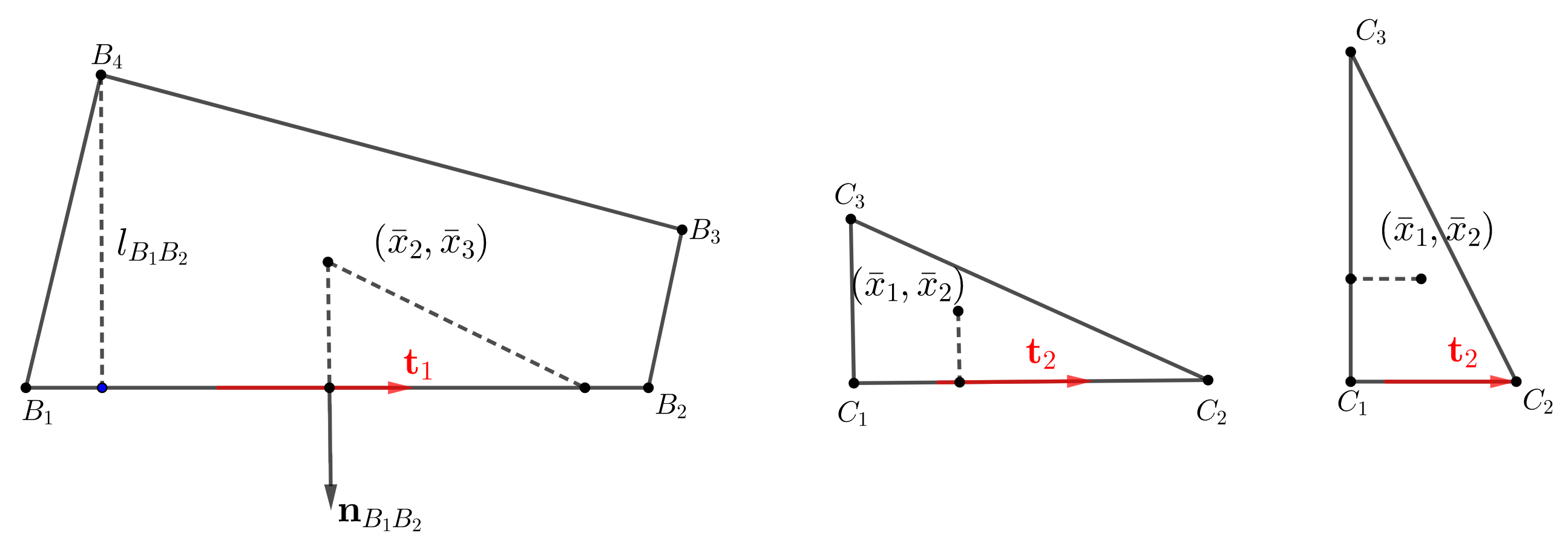}
  \caption{Left: the estimate of $s_h$ for $\bft_1$, and $|s_h(x_2,x_3)|$ is the length of projection of the vector from $(\bar{x}_1,\bar{x}_2)$ to $(x_1,x_2)$ onto $\bfn_{B_1B_2}$, 
  which cannot be greater than $l_{B_1B_2}$. 
  Middle: the estimate of $\tilde{s}_h$ for $\bft_2$ if $h_{C_1C_3}\le h_{C_1C_2}$, and $|\tilde{s}_h(x_1,x_2)|$ is the length of projection of the vector from $(\bar{x}_1,\bar{x}_2)$ to $(x_1,x_2)$ onto $C_1C_3$, 
  which cannot be greater than $h_{C_1C_3}$. 
  Right: the estimate of $\tilde{s}_h$ for $\bft_2$ if $h_{C_1C_3}\ge h_{C_1C_2}$, and $|\tilde{s}_h(x_1,x_2)|$ is the length of projection of the vector from $(\bar{x}_1,\bar{x}_2)$ to $(x_1,x_2)$ onto $C_1C_2$, 
  which cannot be greater than $h_{C_1C_2}$.}
  \label{fig:quad}
\end{figure}


\begin{remark}
\label{rem_tetangle}
The key technique in Lemma \ref{lem_assump_verify4} is to transfer the volume and face integrals into edge integrals, see \eqref{lem_uI_face2_eq2} and \eqref{lem_assump_verify4_eq1}. It heavily relies on that edges of faces cannot be nearly collinear, and respectively, edges of elements cannot be nearly coplanar. In fact, this is the essential meaning of the maximum angle condition, see Lemma \ref{lem_edge_max} for illustration. For the special case of elements in this section, we can find three orthogonal edges. For the elements in Section \ref{sec:fitted_1}, Assumption \hyperref[asp:A3]{(A3)} actually also implies the existence of such three edges. In summary, we shall call an element ``degenerate" if its edges are all are nearly coplanar, which is the non-favorable case in this work.
\end{remark}



 \section{Application III: virtual spaces with discontinuous coefficients}
 \label{sec:IVEM}
 
In this section, we consider the case that a single element contains multiple materials corresponding to different PDE coefficients. 
Henceforth, we assume $\Omega$ is partitioned into two subdomains $\Omega^{\pm}$ by a surface $\Gamma$, called interface, 
and assume $\beta$ in the model problem \eqref{model} is a piecewise constant function: $\beta|_{\Omega^{\pm}} = \beta^{\pm}$,
where the assumption of two subdomains (two materials) is only made for simplicity. Define $v^{\pm}:=v|_{\Omega^{\pm}}$ for any appropriate function $v$. 
Note that $u \in H^1(\Omega)$ and $\beta \nabla u \in \bfH(\text{div};\Omega)$ are not trivial now, but correspond to the jump conditions 
\begin{equation}
\label{jump_cond}
\jump{u}_{\Gamma}:=u^+|_{\Gamma} - u^-|_{\Gamma} = 0 ~~~~~ \text{and} ~~~~~  \jump{\nabla u \cdot \bfn}_{\Gamma}:= \nabla u^+|_{\Gamma} \cdot \bfn - \nabla u^-|_{\Gamma} \cdot \bfn = 0. 
\end{equation}
Generally, the jump conditions make the solutions to interface problems only have a piecewise higher regularity. Accordingly, the local PDEs in \eqref{lift1} to define the virtual spaces involve discontinuous coefficients.
In fact, \eqref{lift1} is a local interface problem whose solutions belonging to $H^1(K)$ \cite{1998ChenZou,2002HuangZou} satisfy the jump conditions in \eqref{jump_cond} on $\Gamma^K$.

Given $D\subseteq \Omega$ intersecting $\Gamma$, define $H^k(D^-\cup D^+) := \{ u\in L^2(D): ~ u|_{D^{\pm}} \in H^k(D) \}$, with an integer $k \ge 0$. For smooth $\Gamma$ and $\partial \Omega$, by \cite{1998ChenZou} the solution $u$ belong to the space
\begin{equation}
  \label{beta_space_2}
H^2_0(\beta;\Omega) = \{ u\in H^1_0(\Omega) \cap H^2(\Omega^-\cup\Omega^+): \beta\nabla u\in \bfH(\text{div};\Omega) \}
\end{equation}
which is slightly modified from \eqref{beta_space_1} due to the discontinuity of $\beta$.
Define the Sobolev extensions $u^{\pm}_E\in H^2_0(\Omega)$ of $u^{\pm}$ from $\Omega^{\pm}$ to $\Omega$. The following boundedness holds \cite{2001GilbargTrudinger}
\begin{equation}
\label{sobolev_ext}
\| u^{\pm}_E \|_{H^2(\Omega)} \le C_{\Omega}  \| u^{\pm} \|_{H^2(\Omega^{\pm})},
\end{equation}
for a constant $C_{\Omega} $ only depending on the geometry of $\Gamma$. 
We further define the norms $\|u\|_{E,k,D} = \|u^+_E\|_{k,D}+\|u^-_E\|_{k,D}$ and $|u|_{E,k,D} = |u^+_E|_{k,D}+|u^-_E|_{k,D}$ to simplify the presentation.

The method analyzed in this section is referred to as the immersed virtual element method (IVEM) developed in \cite{2021CaoChenGuoIVEM,2022CaoChenGuo}. 
It is designed to solve interface problems on unfitted meshes, i.e., an interface surface is allowed to cut the interior of elements, by projecting the virtual spaces onto some non-polynomial and non-smooth spaces, i.e., the IFE spaces, that can capture the conditions in \eqref{jump_cond}. We refer readers to \cite{2020GuoLinZou,2022JiWangChenLi,2015LinLinZhang,2018ZhuangGuo} for various IFE spaces and schemes.
As the focus of this section is to tackle the discontinuous coefficients on unfitted meshes, we shall assume the background mesh cut by the interface is a simple tetrahedral mesh, see Figure \ref{fig:tetinter} for an illustration. 
The assumption is added into \hyperref[asp:A5]{(A5)} below. 
In fact, this is also a widely-used setup in practice, as meshes are not needed to fit the interface \cite{2020AdjeridBabukaGuoLin,2020GuoLin,2020GuoZhang}.

\begin{figure}[h]
\centering
\begin{minipage}{.4\textwidth}
  \centering
  \includegraphics[width=1in]{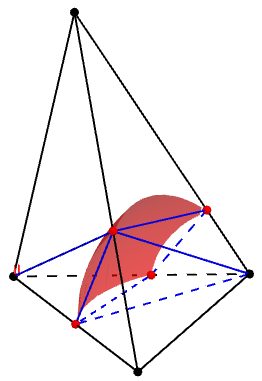}
  \caption{Illustration of an interface element. Being a polyhedron, it has 8 vertices of which 4 are the vertices of $K$ and 4 are the intersecting points with $\Gamma^K_h$. The face triangulation satisfies the }
  \label{fig:tetinter}
\end{minipage}
\begin{minipage}{0.5\textwidth}
  \centering
   \includegraphics[width=2in]{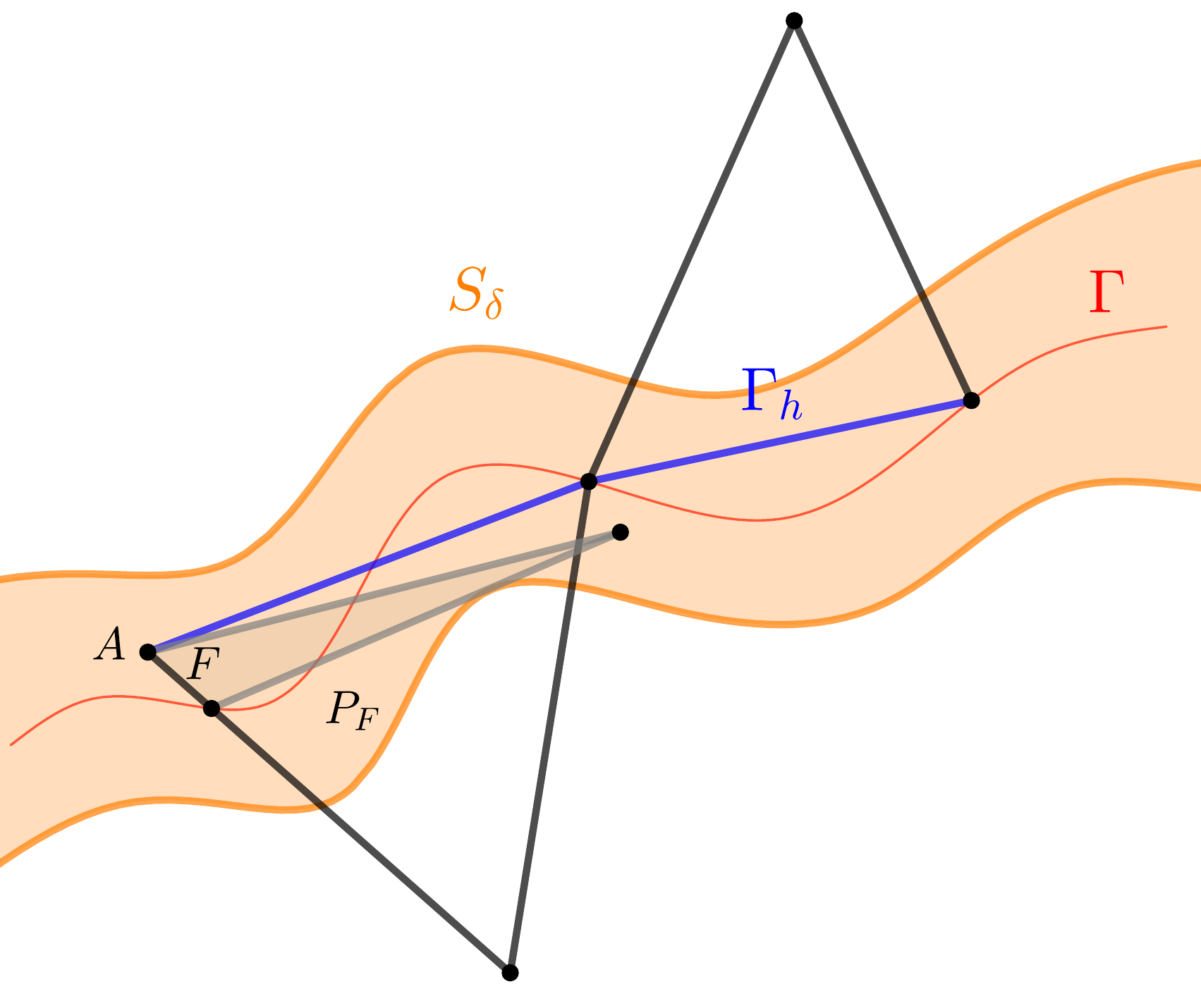}
  \caption{2D Illustration of Assumption \hyperref[asp:A5]{(A5)}. }
  \label{fig:assump3}
  \end{minipage}
\end{figure}

Let us first review some fundamental ingredients of the IVEM. 
Denote the collection of interface elements: $\mathcal{T}^i_h=\{ K\in \mathcal{T}_h: K \cap \Gamma \neq \emptyset \}$. 
As the linear method is used, we let $\Gamma_h$ be a linear approximation to $\Gamma$, where $\Gamma_h$ can be constructed as a linear interpolant of the level-set function of $\Gamma$ on the mesh $\mathcal{T}_h$. 
For each $K\in \mathcal{T}^i_h$ intersecting the interface, we let $\Gamma^K= \Gamma \cap K$ and  $\Gamma^K_h= \Gamma_h \cap K$. 
Then, $K$ can be regarded as a polyhedron whose vertices include the vertices of both $K$ and $\Gamma^K_h$, and short edges and small faces may appear. For $\mathcal{T}_h$ being shape-regular, 
it is easy to see that the face triangulation in Assumption \hyperref[asp:A1]{(A1)} holds, 
and thus we still use $\mathcal{B}_h(\partial K)$ as the trace space. 
We refer readers to Figure \ref{fig:tetinter} for illustration of polyhedron and triangulation.

Let $\beta_h$ be defined with $\Gamma_h$, 
Apparently, polynomial spaces cannot capture the jump behavior across $\Gamma^K_h$ and thus are not suitable for projection. 
Instead, we employ the following linear immersed finite element (IFE) spaces as the projection space. 
A local linear IFE function is a piecewise polynomial space defined below
\begin{equation}
\label{ife_fun_space}
\mathcal{W}_h(K) := \{ v_h|_{K^{\pm}_h}\in \mathcal{P}_1(K^{\pm}_h): v_h\in H^1(K), ~~~ \beta_h \nabla v_h \in \bfH(\text{div};K) \},
\end{equation}
of which the conditions are equivalent to $\jump{v_h}_{\Gamma^K_h}=0$ and $\jump{\beta_h\nabla v_h\cdot\bar{\bfn}_K}_{\Gamma^K_h} = 0$. 
We can derive explicit representation of the IFE functions.
The continuity condition shows that $\nabla v_h$ must be continuous tangentially on $\Gamma^K_h$. 
Namely, for $\bar{\bft}^1_K$ and $\bar{\bft}^2_K$ being two orthogonal unit tangential vectors of $\Gamma^K_h$, there holds $\nabla v^-_h \cdot \bar{\bft}^i_K = \nabla v^+_h \cdot \bar{\bft}^i_K$, $i=1,2$.
With the flux jump condition, we have the following identities:
\begin{equation}
\label{gradv}
\nabla v^-_h = M^- \nabla v^+_h ~~~ \text{and} ~~~ \nabla v^+_h = M^+ \nabla v^-_h, ~~~~ M^- = [\bar{\bft}^1_K,\bar{\bft}^2_K,\beta^-\bar{\bfn}_K]^{-T}[\bar{\bft}^1_K,\bar{\bft}^2_K,\beta^+\bar{\bfn}_K]^T,
\end{equation}
where $M^+=(M^-)^{-1}$.
Define the piecewise constant vector space:
\begin{equation}
\label{Cspace}
\bfP^{\beta}_h(K) = \{  \bfp_h |_{K^{\pm}_h}\in[\mathcal{P}_1(K^{\pm}_h)]^3,~ \bfp_h|_{K^-_h} = M^-\bfp_h|_{K^+_h} \}.
\end{equation}
Therefore, given any point $\bfx_K\in \Gamma^K_h$, the IFE space in \eqref{ife_fun_space} is equivalent to
\begin{equation}
\label{lem_ife_space_eq0}
\mathcal{W}_h(K) = \{ \bfp_h \cdot(\bfx - \bfx_K) + c~:~ \bfp_h \in \bfP^{\beta}_h(K), ~ c\in\mathcal{P}_0(K) \}.
\end{equation}
One can verify that the space in \eqref{lem_ife_space_eq0} is invariant with respect to $\bfx_K$.


It is trivial that $\text{div}(\beta_h \nabla v_h)=0$, $\forall v_h \in \mathcal{W}_h(K)$. 
Thus, the projection in \eqref{Pi_proj} is computable, which is $\beta_h$-weighted different from the standard projection in the previous two cases.
In addition, since $\mathscr{L}_K$ preserves constants as $\nabla \mathcal{P}_0(K)$ always vanishes, and thus Hypothesis \hyperref[asp:H1]{(H1)} holds. 
For Hypothesis \hyperref[asp:H2]{(H2)}, we can employ Assumptions \hyperref[asp:A1]{(A1)} and \hyperref[asp:A2]{(A2)}, 
but as $K$ is shape regular we can also use the standard Poincar\'e inequality \cite{Brenner;Sung:2018Virtual}.
We then need to verify Hypotheses \hyperref[asp:H3]{(H3)}-\hyperref[asp:H6]{(H6)} below, 
with slightly modifying the right-hand sides in \eqref{approxi_VK_Pi}-\eqref{betah_approx} by replacing $\| u \|_{0,\Omega}$ by $\| u \|_{0,\Omega^-\cup\Omega^+}$, due to the regularity. 
Accordingly, the regularity assumptions in Theorems \ref{thm_energy_est} and \ref{thm_u} becomes the space in \eqref{beta_space_2}, 
and the meta-framework developed in Section \ref{sec:unifyframe} is also applicable.
Furthermore, the analysis is standard on non-interface elements, as $\beta_h$ reduces to a single constant. In the following discussion, we focus on interface elements.



Let us first introduce an assumption on the geometric error caused by $\Gamma$ and $\Gamma_h$. 
Let $\Gamma_h$ cut $\Omega$ into two polyhedral subdomains $\Omega^{\pm}_h$ differing from $\Omega^{\pm}$ in a small region which can be defined as $\widetilde{\Omega}_h := (\Omega^+\cap\Omega^-_h)\cup(\Omega^-\cap\Omega^+_h)$ called the \textit{mismatched region}. 
A classical theoretical tool to handle the mismatching region in FEMs is the $\delta$-strip argument developed in \cite{2010LiMelenkWohlmuthZou} recalled below. 
Define a $\delta$-strip:
$
S_{\delta} = \{ \bfx: \text{dist}(\bfx,\Gamma)\le \delta \}.
$
Denote $ K \cap S_{\delta} = K\cap \widetilde{\Omega}_h$ and $\widetilde{F} = F\cap \widetilde{\Omega}_h$ for each element $K$ and face $F$.
Make the following assumption:
\begin{itemize}
  \item[(\textbf{A5})] \label{asp:A5} (The $\delta$-strip condition) $\mathcal{T}_h$ is a shape-regular tetrahedral mesh. 
  On $\mathcal{T}_h$, $\Gamma_h$ is an optimal linear approximation to $\Gamma$ in the sense that
  \begin{equation}
\label{delta_strip_arg}
\widetilde{\Omega}_h \subseteq S_{\delta}, ~~~ \text{for some}~ \delta\lesssim h^2.
\end{equation} 
In addition, assume $S_{\delta}$ satisfies that, for each face $F$ of an element $K$, there is a pyramid $P_F\subseteq S_{\delta}\cap\omega_K$ with $\widetilde{F}$ as its base such that the associated supporting height is $\mathcal{O}(h_K)$. 
\end{itemize} 
\eqref{delta_strip_arg} basically means the optimal geometric accuracy of a linear approximation a surface, which indeed holds for smooth surfaces \cite{2016WangXiaoXu,2020GuoLin}.
A 2D illustration of Assumption \hyperref[asp:A5]{(A5)} is shown in Figure \ref{fig:assump3}. 
\begin{lemma}{\cite[Lemma 2.1]{2010LiMelenkWohlmuthZou}}
\label{lem_delta}
Let $u\in H^1(\Omega^-\cup\Omega^+)$, then there holds
\begin{equation}
\label{lem_delta_eq1}
\| u \|_{L^2(S_{\delta})} \lesssim \sqrt{\delta} \| u \|_{H^1(\Omega^-\cup\Omega^+)}.
\end{equation}
\end{lemma}
With the $\delta$-strip and \cite[Lemma 2.1]{2010LiMelenkWohlmuthZou} recalled below, we can control the error occurring in the mismatched region, and show Hypothesis \hyperref[asp:H6]{(H6)}.
\begin{lemma}
\label{lem_h6_verify}
Let $u\in H^2_0(\beta;\Omega)$. Under Assumption \hyperref[asp:A5]{(A5)}, Hypothesis \hyperref[asp:H6]{(H6)} holds.
\end{lemma}
\begin{proof}
It follows from the definition of $\beta_h$ that $\| \beta \nabla u - \beta_h \nabla u \|_{0,K}\lesssim \|\nabla u \|_{0,K\cap S_{\delta}}$. 
As for the estimates on faces, it only appears on those intersecting with the interface. 
Given an interface face $F$, we consider the pyramid $P_F$ from Assumption \hyperref[asp:A5]{(A5)}, by the trace inequality, there holds
\begin{equation}
\begin{split}
\label{lem_h6_verify_eq1}
h^{1/2}_K \| \beta \nabla u\cdot\bfn -  \beta_h \nabla u \cdot\bfn \|_{0,F} & \lesssim \sum_{s=\pm} \| \nabla u^{s}_E\cdot \bfn \|_{0,\widetilde{F}}  \lesssim \sum_{s=\pm}    |u^{s}_E |_{H^1(P_F)} + h_K | u^{s}_E |_{H^2( P_F)} \\
& \lesssim  |u|_{E,1,S_\delta \cap \omega_K}  + h_K |u|_{E,2,S_\delta \cap \omega_K}.
\end{split}
\end{equation}
Summing \eqref{lem_h6_verify_eq1} over all the interface elements and using Lemma \ref{lem_delta} and \eqref{sobolev_ext} finishes the proof.
\end{proof}

Now, we recall some existing results for the IFE spaces.
\begin{lemma}[Lemma 4.1, \cite{2022CaoChenGuo}]
\label{lem_trace_inequa}
The following trace inequality holds for each $K$:
\begin{equation}
\label{lem_trace_inequa_eq02}
\| \nabla v_h \|_{0, \partial K} \lesssim h^{-1/2}_K \| \nabla v_h \|_{0, K}, ~~~~ \forall v_h\in \mathcal{W}_h(K).
\end{equation}
\end{lemma}
Note that this is non-trivial as $\nabla v_h$ are merely $L^2$ functions. 
The next one concerns interpolation errors gauging by a specially-designed quasi-interpolation:
\begin{equation}
\label{quaInterp_1}
J_K u =
\begin{cases}
      & J_K^-u = \mathrm{P}^1_{\omega_K} u^-_E, ~~~~~~~~~~~~~~~~~~~~~~~~~~~~~~~~~~~~~~~~~~~~~~~~~~~~~ \text{in} ~ \omega^+_K, \\
      & J_K^+u =\mathrm{P}^1_{\omega_K} u^-_E + (\tilde{\beta} - 1) \nabla \mathrm{P}^1_{\omega_K} u^-_E\cdot\bar{\bfn}_K ( \bfx - \bfx_K )\cdot\bar{\bfn}_K, ~~~ \text{in} ~ \omega^-_K,
\end{cases}
\end{equation}
where $\tilde{\beta} = \beta^-/\beta^+$.
One can easily show $J_K u = \bfp_h\cdot(\bfx-\bfx_K) +c$ with $\bfp^-_h  = \nabla  \mathrm{P}^1_{\omega_K} u^-_E $ 
and $\bfp^+_h = \bfp_h^- +(\tilde{\beta}-1)(\bfp^-_h\cdot\bar{\bfn}_K) \bar{\bfn}_K$ 
and $c =  \mathrm{P}^1_{\omega_K} u^-_E(\bfx_K)$, and thus $J_K u$ is an IFE function by \eqref{lem_ife_space_eq0}.
In the following discussion, $J^{\pm}_Ku$ are regarded as polynomials of which each is defined on the entire patch $\omega_K$ instead of just the sub-patches.

\begin{theorem}[Theorem 4.1, \cite{2020GuoLin}]
\label{thm_interp}
Let $u\in H^2_0(\beta;\Omega)$. Then, for every $K\in\mathcal{T}^{i}_h$, there holds
\begin{equation}
\label{thm_interp_eq01}
h^j_K |u^{\pm}_E - J^{\pm}_Ku|_{H^j(\omega_K)} \lesssim h^2_K \|u \|_{E,2,\omega_K}  ,  ~~~~ j=0,1.
\end{equation}
\end{theorem}

With the theorem above, we can estimate the projection errors. 
Similarly, each of $\Pi^{\pm}_Ku$ is regarded as a polynomial defined on the whole patch. 
The key is to estimate $u^{\pm}_E - \Pi^{\pm}_K u$ on the whole element.
\begin{lemma}
\label{lem_beta_proj}
Let $u\in H^2_0(\beta;\Omega)$. Then, for every $K\in\mathcal{T}^{i}_h$, there holds
\begin{equation}
\begin{split}
\label{lem_beta_proj_eq0}
\|\nabla(u^{\pm}_E - \Pi^{\pm}_K u) \|_{0, K} & \lesssim  h_K \|u \|_{E,2,\omega_K} + \| u \|_{E,1,\omega_K \cap S_{\delta}}  .
\end{split}
\end{equation}
\end{lemma}
\begin{proof}
 For simplicity, we only show \eqref{lem_beta_proj_eq0} for $-$ component. 
 By the projection property, 
\begin{equation}
\begin{split}
\label{lem_beta_proj_eq1}
\| \sqrt{\beta_h} \nabla(u^-_E - \Pi^{-}_K u) \|_{L^2(K^-_h)}  \le & \| \sqrt{\beta_h} \nabla (u -  \Pi_K u) \|_{0, K} +  \| \sqrt{\beta_h} \nabla u \|_{E,0, K \cap S_{\delta}} \\
 \le & \| \sqrt{\beta_h} \nabla (u -  J_K u) \|_{0, K} +  \| \sqrt{\beta_h} \nabla u \|_{E,0, K \cap S_{\delta}}  \\
  \le & \sum_{s=\pm}  \|\sqrt{\beta^s_h}\nabla (u^s_E -  J^s_K u) \|_{0, K} +   2 \| \sqrt{\beta_h} \nabla u \|_{E,0, K \cap S_{\delta}}
\end{split}
\end{equation}
which yields \eqref{lem_beta_proj_eq0} on $K^-_h$ by Theorem \ref{thm_interp}. As for $K^+_h$, we note that
\begin{equation}
\begin{split}
\label{lem_beta_proj_eq2}
\|\nabla(u^-_E - \Pi^{-}_K u) \|_{0, K^+_h} &\le \| \nabla(u^-_E - J^-_K u) \|_{0, K^+_h} + \|\nabla(\Pi^{-}_K u - J^-_K u) \|_{0, K^+_h}.
\end{split}
\end{equation}
The first term in the right-hand side above directly follows from \eqref{thm_interp_eq01}. 
For the second term, as $v_h:=\Pi_K u - J_K u$ is an IFE function, by \eqref{gradv} and $\| M^+ \| \lesssim 1$, we obtain
\begin{equation}
\begin{split}
\label{lem_beta_proj_eq3}
 \|  \nabla v^-_h \|_{0, K^+_h}  \lesssim \|  \nabla v^+_h \|_{0, K^+_h} 
 \le  \| \nabla(\Pi^{+}_K u - u^+_E) \|_{0, K^+_h} +  \| \nabla( u^+_E - J^+_K u) \|_{0, K^+_h}
\end{split}
\end{equation}
where the estimation of the first term in the right-hand side above is similar to \eqref{lem_beta_proj_eq1} and the estimate of the second term follows from Theorem \ref{thm_interp}. 
Combining these estimates, we obtain \eqref{lem_beta_proj_eq0}.
\end{proof}

Now, we are ready to examine Hypothesis \hyperref[asp:H3]{(H3)}-\hyperref[asp:H5]{(H5)}.

\begin{lemma}
\label{projection term estimate}
Let $u\in H^2_0(\beta;\Omega)$. Under Assumptions \hyperref[asp:A1]{(A1)}, \hyperref[asp:A2]{(A2)} and \hyperref[asp:A5]{(A5)}, Hypothesis \hyperref[asp:H3]{(H3)} holds. 
\end{lemma}
\begin{proof}
By the definition of projection and integration by parts, we immediately have
\begin{equation}
\label{projection term estimate eq1}
\begin{aligned}
 \| \sqrt{\beta_h} \nabla \Pi_K (u-u_I)\|_{0, K}^2 =& (\beta_h \nabla \Pi_K (u-u_I)\cdot \mathbf{ n},u-u_I)_{\partial K}\\
 \leq& \| \nabla \Pi_K (u-u_I)\cdot \mathbf{ n}\|_{0, \partial K} \| \beta_h (u-u_I) \|_{0, \partial K}.
\end{aligned}
\end{equation}
As $ \Pi_K (u-u_I)$ is an IFE function, the trace inequality in Lemma \ref{lem_trace_inequa} and Lemma \ref{lem_uI_face} lead to the estimate of $\|  \nabla \Pi_K (u-u_I)\|_{0, K}$. 
The estimate of $\| (u-u_I) - \Pi_K(u-u_I) \|_{S_K}$ is similar to \eqref{assump_verify2}.
Summing the estimates over all the interface elements and using Lemma \ref{lem_delta} and \eqref{sobolev_ext} finishes the proof.
\end{proof}

\begin{lemma}
\label{lem_boundflux_err}
Let $u\in H^2_0(\beta;\Omega)$. Under Assumptions \hyperref[asp:A1]{(A1)}, \hyperref[asp:A2]{(A2)} and \hyperref[asp:A5]{(A5)}, Hypothesis \hyperref[asp:H4]{(H4)} holds. 
\end{lemma}
\begin{proof}
\eqref{Pi_approx_eq1} immediately follows from Lemma \ref{lem_beta_proj} and Lemma \ref{lem_delta}. 
As for \eqref{Pi_approx_eq2}, by the triangular inequality, given each face $F\in\mathcal{F}_K$, we have
\begin{equation*}
\begin{split}
\label{lem_boundflux_err_eq1}
\|  \nabla (u-\Pi_K u) \|_{0,F}  \le  \sum_{s=\pm}  \|  \nabla (u^{s}_E-\Pi^{s}_K u) \|_{0,F}   +  \|   \nabla (  u^+_E -  u^-_E )\cdot \mathbf{ n}\|_{0,\widetilde{F}}.
\end{split}
\end{equation*}
The estimate of the first term follows from Lemma \ref{lem_beta_proj} with the classical trace inequality applied on the entire element, 
while the estimate of the second term is similar to \eqref{lem_h6_verify_eq1}. 
Summing the estimates over all the interface elements and using Lemma \ref{lem_delta} and \eqref{sobolev_ext} finishes the proof.
\end{proof}

 \begin{lemma}
\label{lem_boundinterp}
Let $u\in H^2_0(\beta;\Omega)$. Under Assumptions \hyperref[asp:A1]{(A1)}, \hyperref[asp:A2]{(A2)} and \hyperref[asp:A5]{(A5)}, Hypothesis \hyperref[asp:H5]{(H5)} holds. 
\end{lemma}
\begin{proof}
As $K$ is shape regular and $u,\Pi_K u \in H^1(K)$, by the standard Poincar\'e inequality, there holds $\| u - \Pi_K u \|_{0,K}\lesssim h_K \| \nabla( u - \Pi_K u ) \|_{0,K}$ whose estimate then follows from Lemma \ref{lem_beta_proj}. The estimate of $\| u - u_I \|_{0,\partial K}$ follows from Lemma \ref{lem_uI_face}. Summing the estimates over all the interface elements and using Lemma \ref{lem_delta} and \eqref{sobolev_ext} finishes the proof.
\end{proof}

In summary, the numerical solutions admit optimal errors in terms of both the energy and $L^2$ norms by Theorems \ref{thm_energy_est} and \ref{thm_u}, and we refer readers to numerical examples in \cite{2022CaoChenGuo}.


\begin{appendices}

\section{Relation between different geometry assumptions}
\label{appen_lem_tet_maxangle}

\begin{lemma}
\label{lem_tet_maxangle}
Let a polyhedron $D$ be star convex with respect to a ball with the radius $\rho_D$, then the following results hold
\begin{itemize}
\item[(\textbf{G1})] for each $F\in \mathcal{F}_D$, there is a tetrahedron $T\subseteq D$ that has $F$ has one of its faces and the supporting height is greater than $\rho_D$.
\item[(\textbf{G2})] for each $e\in \mathcal{E}_D$, there is a trapezoid $T$ that has $e$ as one of its edges and has the largest inscribed ball of the radius larger than $\rho_D/2$. In addition there is a pyramid $T'$ that has $T$ as its base and the height is $\rho_D$
\item[(\textbf{G3})] for each $\bfx \in \mathcal{N}_h(\partial K)$, there is a shape regular tetrahedron $T\subseteq D$ with the size greater than $\rho_D$ that has $\bfx$ as one of its vertices.
\end{itemize}
\end{lemma}
\begin{proof}
Let $O$ be the center of the ball denoted by $B_D$. (\textbf{G1}) can be simply verified by forming a pyramid that has the base $F$ and the apex $O$ as the distance from $O$ to the plane of F is certainly larger than $\rho_D$. For (\textbf{G2}), $T$ can be chosen as the trapezoid formed by $e$ and the segment passing through $O$ parallel to $e$. Consider another point $P$ on $B_D$ such that $PO$ is penperdicular to $T$, then the tetrahedron formed by $T$ and $P$ fulfills the requirement. (\textbf{G3}) follows from a similar argument.
\end{proof}

\begin{figure}[h]
  \centering
  \begin{minipage}{.47\textwidth}
  \centering
  \includegraphics[width=2.3in]{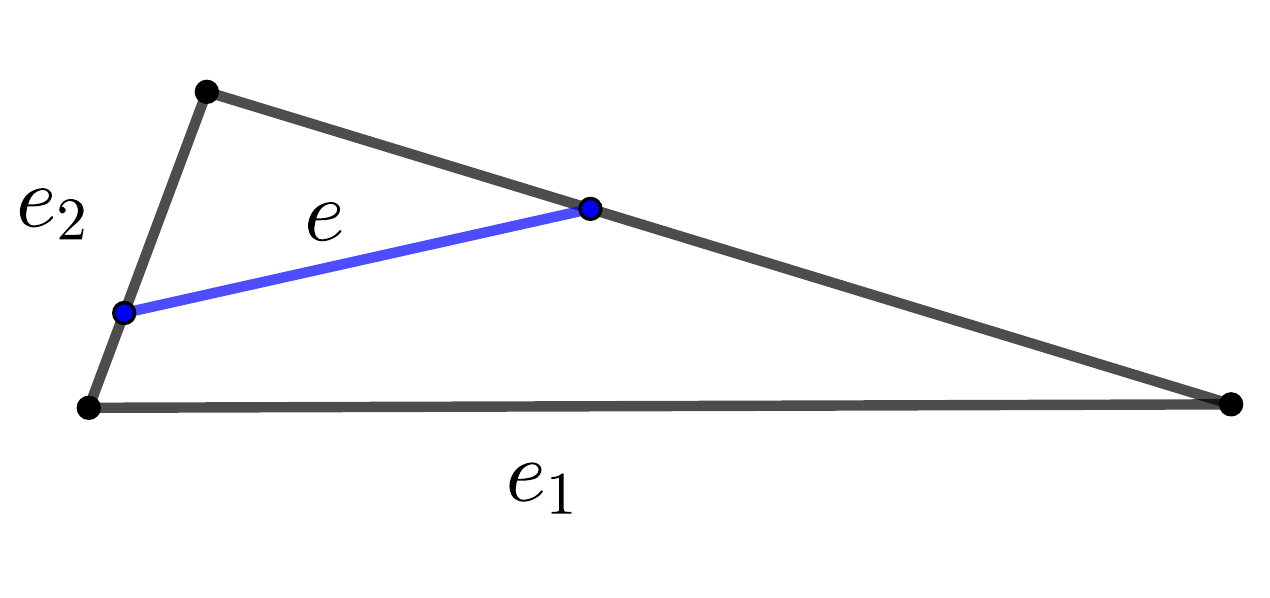}
  \caption{Illustration of the proof of Lemma \ref{lem_edge_max}.}
  \label{fig:ebound}
  \end{minipage}
  ~~~~
    \begin{minipage}{0.47\textwidth}
  \centering
   \includegraphics[width=2.3in]{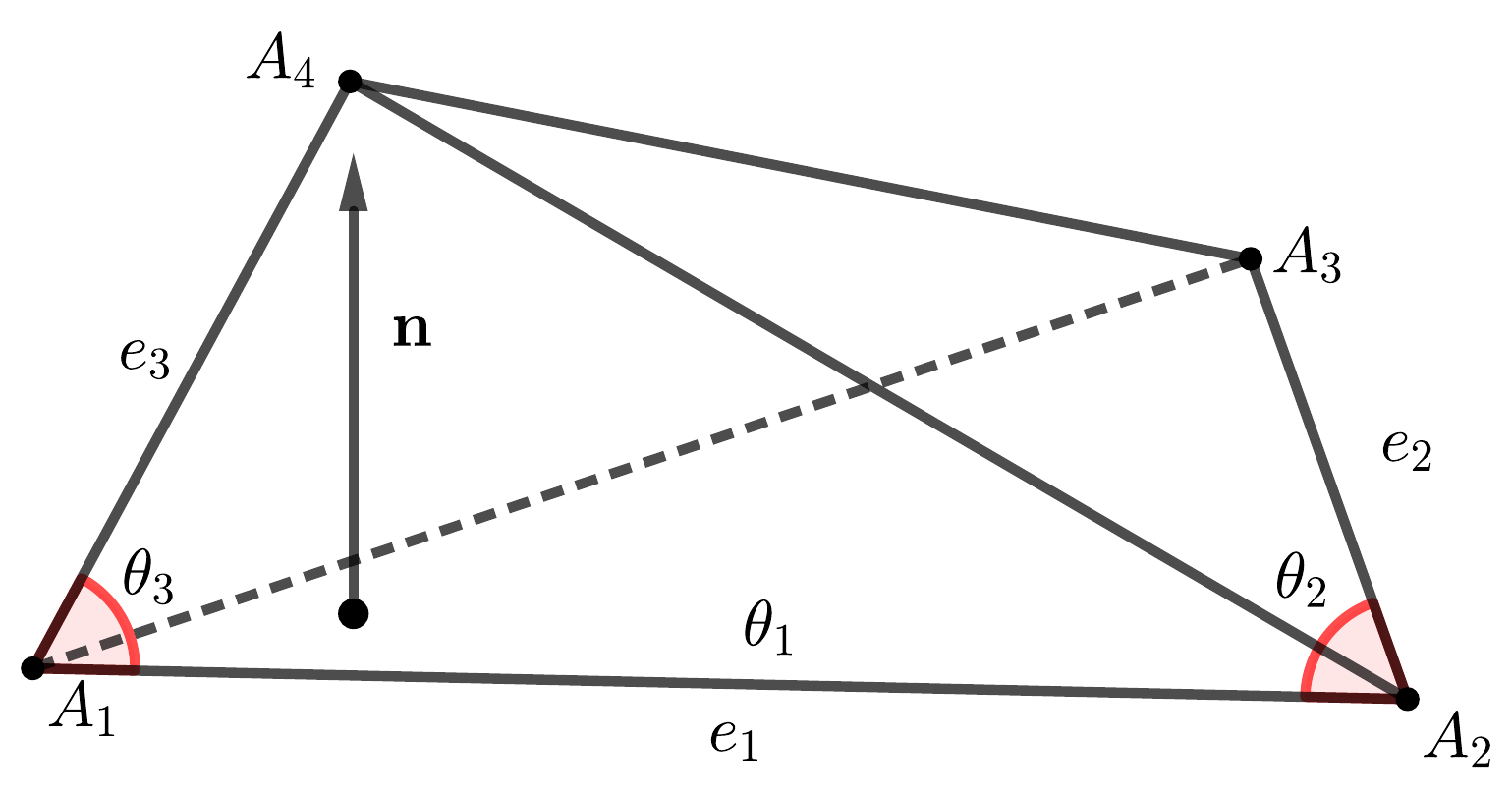}
  \caption{Illustration of the proof of Lemma \ref{lem_regular_edges}.}
  \label{fig:tetra_maxangle}
  \end{minipage}
\end{figure}

\section{Proof of Lemma \ref{lem_A2plus}}
\label{append:lem_A2plus}

We show a stronger version of  Assumption \hyperref[asp:A2]{(A2)} that any two vertices connected by an edge must admit a path satisfying the property in Assumption \hyperref[asp:A2]{(A2)}. 
Call an element \textit{isotropic} if it has the minimum angle $\theta_m$. Consider two vertices $A_1$ and $A_3$ of a triangle $T_1$, as shown in Figure \ref{fig:anisotrop_triangles}.
If $T_1$ is isotropic, then we just choose the path as $A_1$-$A_3$ with $\epsilon = \theta_M/\theta_m$. We focus on $T_1$ being anisotropic.

Case 1. Suppose $\angle A_1A_3A_2$ is the minimum angle of $T_1$, shown by the left plot in Figure \ref{fig:anisotrop_triangles}. 
If $\angle A_1A_3A_2 \rightarrow 0$, by the assumption, one of the two triangles $T_2$ and $T_3$ must be isotropic. Then, one of the paths $A_1$-$A_3$ and $A_1-A_2-A_3$ must have the desired property.  If $\angle A_1A_3A_2$ is also bounded below, and neither of $T_2$ and $T_3$ has the minimum angle $\theta_m$, we can estimate $\angle A_1A_3A_2$ by considering $T_4$. As $T_4$ has the minimum angle $\theta_m$ and has the size greater than $\rho h_{T_1}$, by sine law we know its edges have the minimum length $\sin(\theta_m) \rho h_{T_1} $. So $A_1A_2$ is bounded below by this quantity. As $\angle A_1A_3A_2$ is the minimum angle, $A_1A_2$ is also the edge with the minimum length. Therefore, using the sine law again, we have $\sin(\angle A_1A_3A_2) \ge \rho \sin(\theta_m) \sin(\theta_M)$, which implies that $T_1$ is isotropic with the minimum angle $ \arcsin (\rho \sin(\theta_m) \sin(\theta_M))$. So the path $A_1$-$A_3$ has the desired property.

Case 2. If $\angle A_1A_3A_2$ is not the minimum angle of $T_1$, without loss of generality, we suppose $\angle A_3A_1A_2$ is the minimum angle, shown by the right plot in Figure \ref{fig:anisotrop_triangles}. By the assumption, one of $T_2$ and $T_3$ must isotropic. Similarly, one of the paths $A_1-A_3$ and $A_1-A_2-A_3$ has the desired property.


\section{Estimates regarding maximum angle conditions}

\begin{lemma}
\label{lem_grad}
Given a triangle $T$ with the maximum angle $\theta_M$, then there holds
\begin{equation}
\label{lem_grad_eq0}
\| \nabla v_h \|_{0, T} \le  h^{1/2}_T/\sqrt{2\sin(\theta_M)} \sum_{i=1,2,3} \| \nabla v_h \cdot \bft_i\|_{L^2(e_i)}, ~~~~ \forall v_h\in \mathcal{P}_1(T).
\end{equation}
\end{lemma} 
\begin{proof}
Let $R_T$ be the circumradius of $T$ and $\bft_i$ is a unit tangential vector of the edge $e_i$, $i=1,2,3$.
The cotangent formula \cite{2007Chen} and the law of sines gives
\begin{equation}
\| \nabla v_h \|_{0, T}^2 = R_T \sum_{i=1}^3  \cos(\theta_i) \| \nabla v_h \cdot \bft_i\|_{L^2(e_i)}^2, 
\end{equation}
which leads to the desired result by $R_T  \le h_T/(2\sin(\theta_M))$.
\end{proof}

\begin{lemma}
\label{lem_regular_edges}
Assume a tetrahedron $T$ has the maximum angle condition $\theta_M$. Then, $T$ has three edges (may not share one vertex) such that 
\begin{equation}
\label{lem_regular_edges_eq01}
|\emph{det}(M)|\ge c_m:=\min\{\sqrt{3}/2,\sin(\theta_M) \}\min\{ \cos(\theta_M/2),\sin(\theta_M) \}^2, 
\end{equation}
where $M$ is the matrix formed by the unit direct vectors of these edges. In addition, there holds 
\begin{equation}
\label{lem_regular_edges_eq02}
\frac{c_m}{6}|e_1||e_2||e_3| \le |T| \le \frac{1}{6}|e_1||e_2||e_3|.
\end{equation}
\end{lemma}
\begin{proof}
In $T=A_1A_2A_3A_4$, we first choose the edge $e_1$ such that the associated dihedral angle is the largest one, and without loss of generality, we assume $e_1=A_1A_2$, as shown in Figure \ref{fig:tetra_maxangle}. Let this dihedral angle be $\theta_1$, and let the directional vector of $e_1$ be $\bft_1$. By Lemma 6 in \cite{1992Michal}, we have $\theta_1\in[\pi/3,\theta_M]$. This edge has two neighbor elements $\triangle A_1A_2A_3$ and $\triangle A_1A_2A_4$. Then, we pick the edges $e_2$ and $e_3$ from these two faces such that they have the largest angle from $e_1$ in their faces, denoted by $\theta_2$ and $\theta_3$ respectively. Clearly, we have $\theta_2, \theta_3\in[(\pi-\theta_M)/2,\theta_M]$. Note that $e_2$ and $e_3$ may or may not share the same vertex, but the argument for both the two cases are the same. See Figure \ref{fig:tetra_maxangle} for illustration that they do not share a vertex. Let $\bft_2$ and $\bft_3$, respectively, be the directional vectors of $e_2$ and $e_3$. Let $M$ be the matrix formed by these three vectors. Let $\bfn$ be the norm vector to $\bft_1$ and $\bft_2$. 
Then, the direct calculation shows $\abs{ \bft_3\cdot\bfn} = \cos(\theta_1-\pi/2)\sin(\theta_3)$, and thus
\begin{equation}
\label{lem_regular_edges_eq1}
\abs{\text{det}(M)} = \abs{ (\bft_1 \times \bft_2)\cdot \bft_3 } = \sin(\theta_2) \abs{\bfn \cdot \bft_3} = \sin(\theta_1)\sin(\theta_2)\sin(\theta_3).
\end{equation}
Therefore, we obtain the estimates of $\abs{\text{det}(M)}$ by the upper and lower bounds of $\theta_i$, $i=1,2,3$. 

As for \eqref{lem_regular_edges_eq02}, without loss of generality, we consider the tetrahedron shown in Figure \ref{fig:tetra_maxangle}. Let $l$ be the distance from $A_4$ to the plane $\triangle A_1A_2A_3$. It is easy to see $l =|e_3|\sin(\theta_3)\sin(\theta_1)$. Then, we have
\begin{equation}
\label{lem_regular_edges_eq2}
|T| = |e_1||e_2|\sin(\theta_2)l/6 = |e_1||e_2||e_3| \sin(\theta_1)\sin(\theta_2)\sin(\theta_3)/6
\end{equation}
which yields \eqref{lem_regular_edges_eq02}.
\end{proof}

\begin{lemma}
\label{lem_edge_max}
Given a triangle $T$ with maximum angle $\theta_M$, let $e_1$ and $e_2$ be the two edges of $T$ adjacent to the maximum angle, then each for each segment $e \subseteq T$, there holds
\begin{equation}
\label{lem_edge_max_eq0}
\| \bfv_h\cdot\bft_e \|_{0,e} \lesssim  (\sin(\theta))^{-1/2}  \sum_{i=1,2} \| \bfv_h\cdot\bft_{e_i} \|_{0,e_i}, ~~~ \forall \bfv_h \in \left[ \mathcal{P}_0(T) \right]^2.
\end{equation}
\end{lemma}
\begin{proof}
We first consider a right-angle triangle. Clearly, $\theta_M = \pi/2$. Suppose $e_1$ is on the $x_1$ axis. Let the angle sandwiched by $e$ and $e_1$ be $\theta$. Then, we have $\bft_e = \cos(\theta) \bft_1 +  \sin(\theta) \bft_2$, and thus
\begin{equation}
\begin{split}
\label{lem_edge_max_eq1}
\| \bfv_h\cdot\bft_e \|_{0,e} & = |e|^{1/2} |\bfv_h\cdot\bft_e| \le |e|^{1/2} |\cos(\theta) | |\bfv_h\cdot\bft_1| + |e|^{1/2} |\sin(\theta) | |\bfv_h\cdot\bft_2| \\
& \le  \| \bfv_h\cdot\bft_{e_1} \|_{0,e_1} + \| \bfv_h\cdot\bft_{e_2} \|_{0,e_2},
\end{split}
\end{equation}
where we have used $|e|\cos(\theta) \le |e_1|$ and $|e|\sin(\theta) \le |e_2|$. See Figure \ref{fig:ebound} for illustration. Now, for a general triangle, we consider the affine mapping $\bfx = \mathfrak{F}(\hat{\bfx}) := A\hat{\bfx} = [\bft_1,\bft_2] \hat{\bfx}$. Clearly, there holds $\|A\|_2\lesssim 1$ and $\| A^{-1} \|_2\lesssim (\sin(\theta))^{-1}$.
Then, we obtain from \eqref{lem_edge_max_eq1} that
\begin{equation*}
\begin{split}
\label{lem_edge_max_eq2}
\| \bfv_h \cdot \bft_e \|_{0,e} & = |e|^{1/2} | \bfv_h \cdot \bft_e | \lesssim  |\hat{e}|^{1/2} | (A\bfv_h)\cdot \bft_{\hat{e}} | =  \| (A\bfv_h)\cdot \bft_{\hat{e}} \|_{0,\hat{e}}  \\
& \lesssim   \| (A\bfv_h)\cdot\bft_{\hat{e}_1} \|_{0,\hat{e}_1} + \| (A\bfv_h)\cdot\bft_{\hat{e}_2} \|_{0,\hat{e}_2}  
 \lesssim \| A^{-1} \|^{1/2} \left( |e_1|^{1/2} |\bfv_h\cdot \bft_{e_1}| + |e_2|^{1/2} |\bfv_h\cdot \bft_{e_2}| \right) \\
& \lesssim (\sin(\theta))^{-1/2} ( \| \bfv_h\cdot\bft_{e_1} \|_{0,e_1} + \| \bfv_h\cdot\bft_{e_2} \|_{0,e_2}),
\end{split}
\end{equation*}
which finishes the proof.
\end{proof}

\section{A Poincar\'e-type inequality on anisotropic elements}
\begin{lemma}
\label{lem_poin_anis}
Let $P$ be a convex polyhedron with $F$ being one of its faces, and let $l_F$ be the supporting height of $F$. Assume the projection of $P$ onto the plane containing $F$ is exactly $F$. Then, for $u\in H^1(P),$ there holds
\begin{equation}
\label{lem_poin_anis_eq0}
\| u  \|_{0,P} \lesssim l^{1/2}_F \| u \|^2_{0,F} + l_F \| \nabla u \|_{0,P}.
\end{equation}
\end{lemma}
\begin{proof}
Without loss of generality, we assume that $F$ is on the $x_1x_2$ plane. For each $\bfx = (\xi_1,\xi_2,\xi_3)\in P$, let $\bfx_F= (\xi_1,\xi_2,0)$ be the projection of $\bfx$ onto ${F}$, and let $l(\bfx_F)$ be the height at $\bfx_F$.  We can write $w(\bfx) - w(\bfx_F) = \int_{0}^{\xi_3} \partial_{x_3}w \dd x_3$ and derive
\begin{equation}
\begin{split}
\label{lem_poin_anis_eq1}
\| u \|^2_{0,P} &= \int_{{F}}  \int_0^{l(\bfx_F)}  \left( u(\bfx_F) + \int_{0}^{\xi_3} \partial_{x_3}u \dd x_3 \right)^2 \dd x_3 \dd \bfx_F \\
& \le 2 \int_{F} \int_0^{l(\bfx_F)}  | u(\bfx_F) |^2  \dd x_3 \dd \bfx_F +  \int_{F} \int_0^{l(\bfx_F)} \left( \int_{0}^{\xi_3} \partial_{x_3}u \dd x_3 \right)^2 \dd \xi_3  \dd \bfx_F   \\
& \le 2 l_F \| u \|^2_{0,F} + 2l^2_F \| \nabla u \|^2_{0,P}
\end{split}
\end{equation}
where in the last inequality we have also used H\"older's inequality.
\end{proof}

\end{appendices}

 


\bibliographystyle{siamplain}
\bibliography{vem.bib,RuchiBib.bib}

\end{document}